\documentclass[11pt]{article}
\usepackage{amsmath}
\usepackage{amsthm}
\usepackage{amssymb,amsfonts}
\usepackage{fancyhdr}
\usepackage{hyperref}
\usepackage{latexsym,microtype}
\usepackage{todonotes}
\usepackage{nicefrac}
\usepackage{epsfig,euscript}
\usepackage{stmaryrd}
\usepackage{setspace,mathrsfs}
\usepackage{enumerate}
\usepackage[all]{xypic}
\usepackage{bbm,ifpdf,tikz}
\ifpdf
\usepackage{pdfsync}
\fi

\newtheorem{theorem}{Theorem}[section]

\newtheorem{lemma}[theorem]{Lemma}
\newtheorem{proposition}[theorem]{Proposition}
\newtheorem{corollary}[theorem]{Corollary}

{
\theoremstyle{definition}

\newtheorem{example}[theorem]{Example}

\newtheorem{remark}[theorem]{Remark}
}

\newenvironment{proofmain}{\noindent {\bf Proof of Theorem \ref{main}:}}{\qed \par}

\newcommand{\becircled}{\mathaccent "7017}
\newcommand{\excise}[1]{}
\newcommand{\Spec}{\operatorname{Spec}}

\newcommand{\Ext}{\operatorname{Ext}}

\newcommand{\codim}{\operatorname{codim}}
\renewcommand{\dim}{\operatorname{dim}}

\newcommand{\Sym}{\operatorname{Sym}}
\newcommand{\rk}{\operatorname{rk}}

\newcommand{\Gr}{\operatorname{Gr}}

\renewcommand{\and}{\qquad\text{and}\qquad}

\newcommand{\Z}{\mathbb{Z}}
\newcommand{\Q}{\mathbb{Q}}

\newcommand{\R}{\mathbb{R}}

\newcommand{\bP}{\mathbb{P}}
\newcommand{\bA}{\mathbb{A}}

\newcommand{\Fq}{\mathbb{F}_q}
\newcommand{\Fqb}{\overline{\mathbb{F}}_q}
\newcommand{\Fqs}{\mathbb{F}_{q^s}}
\newcommand{\IC}{\operatorname{IC}}
\renewcommand{\H}{\operatorname{H}}
\newcommand{\IH}{\operatorname{IH}}

\newcommand{\Fr}{\operatorname{Fr}}
\newcommand{\cJ}{\mathcal{J}}
\newcommand{\cI}{\mathcal{I}}
\newcommand{\Ql}{\mathbb{Q}_\ell}
\renewcommand{\a}{\alpha}

\newcommand{\cA}{\mathcal{A}}
\newcommand{\cN}{\mathcal{N}}
\newcommand{\cM}{\mathcal{M}}
\renewcommand{\cL}{\mathcal{L}}
\newcommand{\la}{\lambda}

\newcommand{\Lie}{\operatorname{Lie}}
\newcommand{\gm}{\mathbb{G}_m}
\newcommand{\Aut}{\operatorname{Aut}}

\newcommand{\rxz}{r_{xz}}
\newcommand{\Ih}{\mathscr{I}_{\nicefrac{1}{2}}}
\newcommand{\scrI}{\mathscr{I}}

\newcommand{\cO}{\mathcal{O}}
\newcommand{\hbc}{h^{\operatorname{bc}}}

\begin{document}
\spacing{1.2}
\noindent{\Large\bf The algebraic geometry of Kazhdan-Lusztig-Stanley polynomials}\\

\noindent{\bf Nicholas Proudfoot}\\
Department of Mathematics, University of Oregon,
Eugene, OR 97403\\
njp@uoregon.edu\\

{\small
\begin{quote}
\noindent {\em Abstract.} 
Kazhdan-Lusztig-Stanley polynomials are a combinatorial generalization of Kazhdan-Lusztig polynomials
of Coxeter groups that include $g$-polynomials of polytopes and Kazhdan-Lusztig polynomials of matroids.
In the cases of Weyl groups, rational polytopes, and realizable matroids, one can count points over finite fields on flag varieties,
toric varieties, or reciprocal planes to obtain cohomological interpretations of these polynomials.  We survey these results
and unite them under a single geometric framework.
\end{quote} }

\section{Introduction}
The original definition of Kazhdan-Lusztig polynomials for a Coxeter group 
involves the relationship between two bases for the Hecke algebra, however the polynomials are  
characterized by a purely combinatorial recursion involving intervals in the Bruhat poset \cite[Equation (2.2.a)]{KL79}.
Stanley later generalized this recursive definition, replacing the Bruhat poset with an arbitrary locally 
graded poset \cite[Definition 6.2(b)]{Stanley-h}.  Stanley's main motivation was the observation that the 
$g$-polynomial of a polytope, which he introduced in \cite{Stanley-IC},
arises very naturally in this way \cite[Example 7.2]{Stanley-h}.  Brenti went on to generalize this definition slightly further to
weakly ranked posets, and dubbed the corresponding polynomials {\bf Kazhdan-Lusztig-Stanley polynomials}
\cite{Brenti-twisted}.  More examples have been studied since then, including Kazhdan-Lusztig polynomials
of matroids, which were introduced in \cite{EPW} and have been the subject of much recent research.

In a sequel to their first paper, Kazhdan and Lusztig 
proved that their polynomials can be interpreted
as Poincar\'e polynomials for the stalk cohomology groups of the intersection cohomology sheaves of Schubert varieties
\cite[Theorem 4.3]{KL80}.  The idea of the proof is that the combinatorial recursion for the polynomials
is precisely the recursion for the Poincar\'e polynomials that one obtains by applying the Lefschetz fixed point formula
to the Frobenius automorphism of certain subvarieties of the flag variety.  This technique has subsequently been imported
to the study of other classes of Kazhdan-Lusztig-Stanley polynomials, including 
Kazhdan-Lusztig polynomials of affine Weyl groups \cite[Section 11]{Lusztig-singularities},
$g$-polynomials of rational polytopes \cite[Theorem 6.2]{DL} and \cite[Theorem 1.2]{Fieseler},
$h$-polynomials of broken circuit complexes of a rational hyperplane arrangements
\cite[Theorem 4.3]{PW07}, and Kazhdan-Lusztig polynomials of hyperplane arrangements \cite[Theorem 3.10]{EPW}.

While the various instances of the aforementioned Lefschetz argument are all based on the same idea, they all involve
a rather messy induction, and it can be difficult to determine exactly what ingredients are needed to make the argument work.
The purpose of this document is to do exactly that.  After reviewing the combinatorial theory of Kazhdan-Lusztig-Stanley polynomials
(Section \ref{sec:combinatorics}), we lay out a basic geometric framework for interpreting these polynomials as Poincar\'e polynomials
of stalks of intersection cohomology sheaves on a stratified variety (Section \ref{sec:geometry}).  In particular, we show that each
of the aforementioned results can be obtained as an application of our general machine (Section \ref{sec:examples}) without having
to redo the inductive argument each time.\\

Though our main purpose is to survey and unify various old results, there is one new concept that we introduce
and study here.  When defining Kazhdan-Lusztig-Stanley polynomials, there is a left
versus right convention that appears in the definition.
The left Kazhdan-Lusztig-Stanley polynomials for a weakly graded poset $P$ coincide with the right Kazhdan-Lusztig-Stanley
polynomials for the opposite poset $P^*$ (Remark \ref{opposite}).  In particular, since the Bruhat poset of a finite Coxeter group
is self-opposite and the face poset of a polytope is opposite to the face poset of the dual polytope, the left/right issue (while at
times confusing) is not so important.  The same statement is not true of the lattice of flats of a matroid, and indeed the right
Kazhdan-Lusztig-Stanley polynomials of a matroid are interesting while the left ones are trivial (Example \ref{ex:char}).
We introduce a class of polynomials called {\bf \boldmath{$Z$}-polynomials} (Section \ref{sec:Z}) that depend on both the left and right
Kazhdan-Lusztig-Stanley polynomials.  
In the case of the lattice of flats of a matroid, these polynomials coincide with the polynomials
introduced in \cite{PXY}.  

Under certain assumptions, we use another Lefschetz argument to interpret our $Z$-polynomials a Poincar\'e polynomials
for the global intersection cohomology of the closure of a stratum in our stratified variety.  In particular, in the case of the Bruhat
poset of a Weyl group, the $Z$-polynomials are intersection cohomology Poincar\'e polynomials of Richardson varieties
(Theorem \ref{richardson}); in the case of the lattice of flats of a hyperplane arrangements, they are intersection cohomology
Poincar\'e polynomials of arrangement Schubert varieties (Theorem \ref{Z-arr}); and in the case of the affine Grassmannian,
they are intersection cohomology Poincar\'e polynomials of closures of Schubert cells (Corollary \ref{Z-affine}).

It would be interesting to know whether the $Z$-polynomial of a rational polytope has a cohomological interpretation in terms of toric
varieties.  These polynomials are closely related to a family of polynomials defined by Batyrev and Borisov (Remark \ref{BB}),
but they are not quite the same.

\subsection{Things that this paper is not about}
There are many interesting questions about Kazhdan-Lusztig-Stanley polynomials that we will mention briefly here
but not address in the main part of the paper.
\begin{itemize}
\item By giving a cohomological interpretation of a Kazhdan-Lusztig-Stanley polynomial, one can infer that it has non-negative
coefficients.  There is a rich history of pursuing the non-negativity of certain classes Kazhdan-Lusztig polynomials in the absence of a geometric
interpretation.  This was achieved by Elias and Williamson for Kazhdan-Lusztig polynomials of Coxeter groups that are not Weyl groups
\cite[Conjecture 1.2(1)]{EW14} and by Karu for polytopes that are not rational \cite[Theorem 0.1]{Karu} (see also \cite[Theorem 2.4(b)]{Braden-CICF}).
Braden, Huh, Matherne, Wang, and the author are working to prove an analogous theorem for matroids that are not realizable by hyperplane arrangements.
\item For many specific classes of Kazhdan-Lusztig-Stanley polynomials, it is interesting to ask what polynomials can arise.
Polo proved that any polynomial with non-negative coefficients and constant term 1 is equal to a 
Kazhdan-Lusztig polynomial for a symmetric group \cite{Polo}.
In contrast, the $g$-polynomial of a polytope cannot have internal zeros \cite[Theorem 1.4]{Braden-CICF}.
If the polytope is simplicial, then the sequence of coefficients is an M-sequence \cite{S1}, and this is conjecturally the case
for all polytopes; see \cite[Section 1.2]{Braden-CICF} for a discussion of this conjecture. 
Kazhdan-Lusztig polynomials of matroids are conjectured to always be 
log concave with no internal zeros \cite[Conjecture 2.5]{EPW} and even real-rooted \cite[Conjecture 3.2]{kl-survey},
and a similar conjecture has been made for $Z$-polynomials of matroids \cite[Conjecture 5.1]{PXY}.
\item Classical Kazhdan-Lusztig polynomials were originally defined in terms of the Kazhdan-Lusztig basis for the Hecke algebra.
More generally, Du defines the notion of an {\bf IC basis} for a free $\Z[t,t^{-1}]$-module equipped with an involution \cite{Du}, 
and Brenti proves that this notion is essentially equivalent to the theory of Kazhdan-Lusztig-Stanley polynomials \cite[Theorem 3.2]{Brenti-IC}.
Multiplication in the Hecke algebra is compatible with the involution, which Brenti shows is a very special property
\cite[Theorem 4.1]{Brenti-IC}.  Furthermore, the structure constants for multiplication in the Kazhdan-Lusztig basis 
of the Hecke algebra are positive \cite[Conjecture 1.2(2)]{EW14}, and Du asks whether this holds in some greater generality \cite[Section 5]{Du}.
In the case of Kazhdan-Lusztig polynomials of matroids, a candidate algebra structure was described and positivity was conjectured \cite[Conjecture 4.2]{EPW},
but that conjecture turned out to be false (see Section 4.6 of the arXiv version).  It is unclear whether this conjecture could be salvaged
by changing the definition of the algebra structure, or more generally when a particular collection of Kazhdan-Lusztig-Stanley polynomials
comes equipped with a nice algebra structure on its associated module.
\end{itemize}

\vspace{\baselineskip}
\noindent
{\em Acknowledgments:}
This work was greatly influenced by conversations with many people, including Sara Billey,
Tom Braden, Ben Elias, Jacob Matherne, Victor Ostrik, Richard Stanley, Minh-Tam Trinh, 
Max Wakefield, Ben Webster, Alex Yong, and Ben Young.  The author is also grateful to the referee
for many helpful comments.
The author is supported by NSF grant DMS-1565036.

\section{Combinatorics}\label{sec:combinatorics}
We begin by reviewing the combinatorial theory of Kazhdan-Lusztig-Stanley polynomials,
which was introduced in \cite[Section 6]{Stanley-h} and further developed in \cite{Dyer, Brenti-twisted, Brenti-IC}.
We also introduce $Z$-polynomials (Section \ref{sec:Z}) and study their basic properties.

\subsection{The incidence algebra}
Let $P$ be a poset.
We say that $P$ is {\bf locally finite} if, for all $x\leq z\in P$,
the set $$[x,z] := \{y\in P\mid x\leq y\leq z\}$$ is finite.  Let $$I(P) := \prod_{x\leq y} \Z[t].$$
For any $f\in I(P)$ and $x<y\in P$, let $f_{xy}(t)\in \Z[t]$ denote the corresponding component of $f$.
If $P$ is locally finite, then $I(P)$ admits a ring structure with product given by convolution:
$$(fg)_{xz}(t) := \sum_{x\leq y\leq z} f_{xy}(t)g_{yz}(t).$$
The identity element is the function $\delta\in I(P)$ with the property that $\delta_{xy} = 1$ if $x=y$ and 0 otherwise.

Let $r\in I(P)$ be a function satisfying the following conditions:
\begin{itemize} 
\item $r_{xy}\in\Z\subset\Z[t]$ for all $x\leq y\in P$ (we will refer to $r_{xy}(t)$ simply as $r_{xy}$)
\item if $x< y$, then $r_{xy}>0$
\item if $x\leq y\leq z$, then $r_{xy} + r_{yz} = r_{xz}$.
\end{itemize}
Such a function is called a {\bf weak rank function} \cite[Section 2]{Brenti-twisted}.  
We will use the terminology {\bf weakly ranked poset}
to refer to a locally finite poset equipped with a weak rank function, and we will suppress $r$ from the notation
when there is no possibility for confusion.

For any weakly ranked poset $P$, let $\scrI(P)\subset I(P)$ denote the subring of functions $f$ with the property that 
the degree of $f_{xy}(t)$ is less than or equal to $r_{xy}$ for all $x\leq y$.
The ring $\scrI(P)$ admits an involution $f\mapsto\bar f$ defined by the formula
$$\bar f_{xy}(t) := t^{r_{xy}}f_{xy}(t^{-1}).$$

\begin{lemma}\label{inverses}
An element $f\in I(P)$ has an inverse (left or right) if and only if $f_{xx}(t) = \pm 1$ for all $x\in P$.
In this case, the left and right inverses are unique and they coincide.  If $f\in\scrI(P)\subset I(P)$ is invertible, then $f^{-1}\in\scrI(P)$.
\end{lemma}

\begin{proof}
An element $g$ is a right inverse to $f$ if and only if $g_{xx}(t) = f_{xx}(t)^{-1}$ and 
$$f_{xx}(t) g_{xz}(t) = - \sum_{x<y\leq z} f_{xy}(t)g_{yz}(t)$$
for all $x<z$.  The first equation has a solution if and only if $f_{xx}(t) = \pm 1$, 
in which case the second equation also has a unique solution.
If $f\in\scrI(P)$, it is clear that $g\in\scrI(P)$, as well.  The argument for left inverses is identical, 
so it remains only to show that left and right inverses coincide.

Let $g$ be right inverse to $f$.  Then $g$ is also left inverse to some function, which we will denote $h$.  We then have
$$f = f\delta = f(gh) = (fg)h = \delta h = h,$$
so $g$ is left inverse to $f$, as well.
\end{proof}

\subsection{Right and left KLS-functions}
An element $\kappa\in\scrI(P)$ is called a {\bf \boldmath{$P$}-kernel} if $\kappa_{xx}(t) = 1$ for all $x\in P$ and $\kappa^{-1} = \bar \kappa$.
Let $$\Ih(P) := \Big\{f\in \scrI(P)\;\Big{|}\; \text{$f_{xx}(t) = 1$ for all $x\in P$ and $\deg f_{xy}(t)<r_{xy}/2$ for all $x<y\in P$}\Big\}.$$
Various versions of the following theorem appear in \cite[Corollary 6.7]{Stanley-h}, \cite[Proposition 1.2]{Dyer}, and 
\cite[Theorem 6.2]{Brenti-twisted}.

\begin{theorem}\label{thm:KL}
If $\kappa\in\scrI(P)$ is a P-kernel, there exists a unique pair of functions $f,g\in\Ih(P)$ such that
$\bar f = \kappa f$ and $\bar g = g\kappa$.
\end{theorem}

\begin{proof}
We will prove existence and uniqueness of $f$; the proof for $g$ is identical.
Fix elements $x<w\in P$, and suppose that $f_{yw}(t)$ has been defined for all $x<y\leq w$.  
Let $$Q_{xw}(t) := \sum_{x<y\leq w} \kappa_{xy}(t)f_{yw}(t) \in \Z[t].$$
The equation $\bar f = \kappa f$ for the interval $[x,w]$ translates to $$\bar f_{xw}(t) - f_{xw}(t) = Q_{xw}(t).$$
It is clear that there is at most one polynomial $f_{xw}(t)$ of degree strictly less than $r_{xw}/2$ satisfying this equation.
The existence of such a polynomial is equivalent to the statement $$t^{r_{xw}}Q_{xw}(t^{-1}) = -Q_{xw}(t).$$
To prove this, we observe that
\begin{eqnarray*}
t^{r_{xw}}Q_{xw}(t^{-1}) &=& t^{r_{xw}}\sum_{x<y\leq w} \kappa_{xy}(t^{-1})f_{yw}(t^{-1})\\
&=& \sum_{x<y\leq w} t^{r_{xy}} \kappa_{xy}(t^{-1})t^{r_{yw}} f_{yw}(t^{-1})\\
&=& \sum_{x<y\leq w} \bar \kappa_{xy}(t)\bar f_{yw}(t)\\
&=& \sum_{x<y\leq w} \bar \kappa_{xy}(t) (\kappa f)_{yw}(t)\\
&=& \sum_{x<y\leq w} \bar \kappa_{xy}(t) \sum_{y\leq z\leq w} \kappa_{yz}(t) f_{zw}(t)\\
&=& \sum_{x<y\leq z\leq w} \bar \kappa_{xy}(t)\kappa_{yz}(t) f_{zw}(t)\\
&=& \sum_{x<z\leq w} f_{zw}(t) \sum_{x<y\leq z}\bar \kappa_{xy}(t)\kappa_{yz}(t)\\
&=& \sum_{x<z\leq w} f_{zw}(t) \Big((\bar \kappa \kappa)_{xz}(t) - \kappa_{xz}(t)\Big)\\
&=& - \sum_{x<z\leq w} \kappa_{xz}(t)f_{zw}(t)\\
&=& - Q_{xw}(t).
\end{eqnarray*}
Thus there is a unique choice of polynomial $f_{xw}(t)$ consistent with the equation $\bar f = \kappa f$ on the interval $[x,w]$.
\end{proof}

\begin{remark}\label{warning!}
Stanley \cite{Stanley-h} works only with the function $g$, as does Brenti in \cite{Brenti-twisted},
while Brenti later switches conventions and works with the function $f$ in \cite{Brenti-IC} 
(though he notes in a footnote that both functions exist). 
Dyer \cite{Dyer} defines versions of both functions, but with normalizations that differ from ours.

Brenti refers to $g$ in \cite{Brenti-twisted} and $f$ in  \cite{Brenti-IC} as the Kazhdan-Lusztig-Stanley function associated with $\kappa$.
We will refer to $f$ as the {\bf right Kazhdan-Lusztig-Stanley function}
associated with $\kappa$, and to $g$ as the {\bf left Kazhdan-Lusztig-Stanley function}
associated with $\kappa$.
For any $x\leq y$, we will refer to the polynomial $f_{xy}(t)$ or $g_{xy}(t)$ as a (right or left)
{\bf Kazhdan-Lusztig-Stanley polynomial}.  We will write KLS as an abbreviation for Kazhdan-Lusztig-Stanley.
\end{remark}

\begin{remark}\label{opposite}
Given a locally finite weakly graded poset $P$, let $P^*$ denote the {\bf opposite} of $P$, which means 
that $y\leq x$ in $P^*$ if and only if $x\leq y$ in $P$, in which case $r^*_{yx} = r_{xy}$.
For any function $f\in\scrI(P)$, define $f^*\in\scrI(P^*)$ by putting $f^*_{yx}(t) := f_{xy}(t)$ for all $x\leq y\in P$.
If $\kappa$ is a $P$-kernel with right KLS-function $f$ and left KLS-function $g$, then $\kappa^*$ is a $P^*$-kernel
with left KLS-function $f^*$ and right KLS-function $g^*$.  Thus one can go between left and right KLS-polynomials by reversing the order on the poset.
\end{remark}

It will be convenient for us to have a converse to Theorem \ref{thm:KL}.  A version of this proposition appears in \cite[Theorem 6.5]{Stanley-h}.

\begin{proposition}\label{prop:R}
Suppose that $f\in \Ih(P)$.
Then 
\begin{enumerate}
\item $f$ is invertible.
\item $\bar f f^{-1}$ is a $P$-kernel with $f$ as its associated right KLS-function.
\item $f^{-1}\bar f$ is a $P$-kernel with $f$ as its associated left KLS-function.
\end{enumerate}
\end{proposition}

\begin{proof}
By Lemma \ref{inverses}, $f$ is invertible.  We have $(\bar f f^{-1})^{-1} = f\bar{f}^{-1} = \overline{\bar f f^{-1}}$, so $\bar f f^{-1}$ is a $P$-kernel.
Since $\bar f = \bar f (f^{-1} f) = (\bar f f^{-1}) f$, the uniqueness part of Theorem \ref{thm:KL} tells us that $f$ is equal to the associated right KLS-function.  
The last statement follows similarly.
\end{proof}

\subsection{The \boldmath{$Z$}-function}\label{sec:Z}
We will call a function $Z\in\scrI(P)$ {\bf symmetric} if $\bar Z = Z$.
Let $\kappa$ be a $P$-kernel with right KLS-function $f$ and left KLS-function $g$.  Let $Z := g\kappa f\in\scrI(P)$;
we will refer to $Z$ as the {\bf \boldmath{$Z$}-function} associated with $\kappa$, and to each $Z_{xy}(t)$ as a {\bf \boldmath{$Z$}-polynomial}.

\begin{proposition}\label{prop:Z}
We have $Z = \bar g f = g\bar f$.  In particular, $Z$ is symmetric.
\end{proposition}

\begin{proof}
Since $\bar g = g\kappa$, we have $Z = g\kappa f = \bar g f$.  
Since $\bar f = \kappa f$, we have $Z = g\kappa f = g\bar f$.
\end{proof}

We have the following converse to Proposition \ref{prop:Z}.

\begin{proposition}\label{it was g}
Suppose that $f,g\in\Ih(P)$.  Then $f$ and $g$ are the right and left KLS-functions
for a single $P$-kernel $\kappa$ if and only if $\bar g f$ is symmetric.
\end{proposition}

\begin{proof}
Let $\kappa_f := \bar f f^{-1}$ and $\kappa_g := g^{-1}\bar g$.  By Proposition \ref{prop:R},
$f$ is the right KLS-function of $\kappa_f$ and $g$ is the left KLS-function of $\kappa_g$.
Then $\bar g f = g\kappa_g f$ and $g\bar f = g\kappa_f f$.  Multiplying on the left by $g^{-1}$
and on the right by $f^{-1}$, we see that these two functions are the same if and only if $\kappa_f = \kappa_g$.
\end{proof}

The following version of Proposition \ref{it was g} will be useful in Section \ref{sec:Z-geom}.
It allows us to relax both the symmetry assumption and the conclusion of Proposition \ref{it was g}.

\begin{proposition}\label{weak g}
Let $\kappa$ be a $P$-kernel, and let $f,g\in\Ih(P)$ be the associated 
right and left KLS-functions.  Suppose we are given $x\in P$ and
$h\in\Ih(P)$ such that, for all $z\geq x$, we have $(\bar h f)_{xz}(t) = (h\bar f)_{xz}(t)$.
Then for all $z\geq x$, $h_{xz}(t) = g_{xz}(t)$.
\end{proposition}

\begin{proof}
We proceed by induction on $r_{xz}$.  When $z=x$, we have $h_{xx}(t) = 1 = g_{xx}(t)$.
Now assume that the statement holds for all $y$ such that $x\leq y < z$.
We have
$$\sum_{x\leq y\leq z} \bar g_{xy}(t)f_{yz}(t) =
(\bar g f)_{xz}(t) = (g\bar f)_{xz}(t) = \sum_{x\leq y\leq z} g_{xy}(t)\bar f_{yz}(t)$$
and
$$\sum_{x\leq y\leq z} \bar h_{xy}(t)f_{yz}(t) =
(\bar h f)_{xz}(t) = (h\bar f)_{xz}(t) = \sum_{x\leq y\leq z} h_{xy}(t)\bar f_{yz}(t).$$
Subtracting these two equations and applying our inductive hypothesis, we have
$$\bar g_{xz}(t) - \bar h_{xz}(t) = t^{r_{xz}}\big(g_{xz}(t) - h_{xz}(t)\big).$$
Since $\deg(g_{xz}-h_{xz}) < r_{xz}/2$, this implies that $g_{xz}(t) = h_{xz}(t)$.
\end{proof}

\begin{proposition}\label{Z dual}
Let $\kappa\in\scrI(P)$ be a $P$-kernel, and let $P^*$ be the opposite of $P$.
Then $Z^*\in\scrI(P^*)$ is the $Z$-polynomial associated with the $P^*$-kernel $\kappa^*$.
\end{proposition}

\begin{proof}
By Remark \ref{opposite}, the left KLS-polynomial associated with $\kappa^*$ is $f^*$,
and the right KLS-polynomial is $g^*$.  Thus the $Z$-polynomial is $f^*\kappa^*g^* = (g\kappa f)^* = Z^*$.
\end{proof}

\begin{remark}
Let $\kappa$ be a $P$-kernel with right KLS-function $f$, left KLS-function $g$ and $Z$-function $Z$.
Proposition \ref{prop:R} says that, if you know $f$ or $g$, you can compute $\kappa$.  Similarly, we observe
that if you know $Z$, you can compute $f$ and $g$, and therefore $\kappa$.  
This can be proved inductively.  Indeed, assume that we can compute $f$ and $g$ on any interval strictly contained in $[x,z]$.
Then we have $$Z_{xz}(t) = \sum_{x\leq y\leq z} \bar g_{xy}(t) f_{yz}(t) = f_{xz}(t) + \bar g_{xz}(t) + \sum_{x< y< z} \bar g_{xy}(t) f_{yz}(t),$$
and therefore \begin{equation}\label{recover Z}f_{xz}(t) + \bar g_{xz}(t) = Z_{xz}(t) - \sum_{x< y< z} \bar g_{xy}(t) f_{yz}(t).\end{equation}
By our inductive hypothesis, we can compute the right-hand side, which determines the left-hand side.
Since $f,g\in\Ih(P)$, this determines $f_{xz}(t)$ and $g_{xz}(t)$ individually.

On the other hand, it is {\bf not} true that every symmetric function $Z\in\scrI(P)$ with $Z_{xy}(0) = 1$ for all $x\leq y\in P$
is the $Z$-function associated with some $P$-kernel.  This is because Equation \eqref{recover Z} cannot be solved
if $r_{xz}$ is even and the coefficient of $t^{r_{xz}/2}$ on the right hand side is nonzero.
\end{remark}

\subsection{Alternating kernels}
Given a function $h\in \scrI(P)$, we we define $\hat h\in\scrI(P)$ by the formula $\hat h_{xy}(t) := (-1)^{r_{xy}} h_{xy}(t)$.
The map $h\mapsto\hat h$ is an involution of the ring $\scrI(P)$ that commutes with the involution $h\mapsto\bar h$.
We will say that $h$ is {\bf alternating} if $\bar h = \hat h$.
A version of the following result appears in \cite[Corollary 8.3]{Stanley-h}.

\begin{proposition}\label{antiK}
Let $\kappa\in\scrI(P)$ be an alternating $P$-kernel, and let $f,g\in\Ih(P)$ be the associated right and left KLS-functions.
Then $\hat g = f^{-1}$ and $\hat f = g^{-1}$.
\end{proposition}

\begin{proof}
Since $\bar g = g\kappa$, we have $\hat{\bar{g}} = \hat g \hat \kappa = \hat g\bar\kappa$.
Then $$\overline{\hat g f} = \bar{\hat{g}} \bar f = \hat{\bar{g}}\bar f = \hat g\bar\kappa\kappa f = \hat g f,$$
thus $\hat g f$ is symmetric.  However, since $f,g\in\Ih(P)$, we have $\deg (\hat g f)_{xy}(t)< r_{xy}/2$ for all $x<y$,
so this implies that $(\hat g f)_{xy}(t)=0$ for all $x<y$.  On the other hand, $(\hat g f)_{xx}(t)=\hat g_{xx}(t)f_{xx}(t) = 1$.
Thus $\hat g f = \delta$, and therefore $\hat g = f^{-1}$.  The second statement follows immediately.
\end{proof}

\subsection{Examples}
We now discuss a number of examples of $P$-kernels along their associated KLS-functions and $Z$-functions.
All of these examples will be revisited in Section \ref{sec:examples}.

\begin{example}\label{ex:coxeter}
Let $W$ be a Coxeter group, equipped with the Bruhat order and the rank function given by the length of an element of $W$.
The classical {\bf \boldmath{$R$}-polynomials} $\{R_{vw}(t)\mid v\leq w\in W\}$ form a $W$-kernel, and the classical
{\bf Kazhdan-Lusztig polynomials} $\{f_{xy}(t)\mid v\leq w\in W\}$ are the associated right KLS-polynomials.
These polynomials were introduced by Kazhdan and Lusztig \cite{KL79}, and they were one of the main motivating examples in Stanley's work \cite[Example 6.9]{Stanley-h}.

If $W$ is finite, then there is a maximal element $w_0\in W$, and left multiplication by $w_0$ defines
an order-reversing bijection of $W$ with the property that, if $v\leq w$, then $R_{vw}(t) = R_{(w_0w)(w_0v)}(t)$
\cite[Lemma 11.3]{Lusztig-unequal}.  It follows from Remark \ref{opposite} that $g_{vw}(t) = f_{(w_0w)(w_0v)}(t)$.
In addition, $R$ is alternating \cite[Lemma 2.1(i)]{KL79}, hence Proposition \ref{antiK} tells us that $\hat g = f^{-1}$ and $\hat f = g^{-1}$.
\end{example}

\begin{example}\label{ex:char}
Let $P$ be any locally finite weakly ranked poset.  Define $\zeta\in\scrI(P)$ by the formula
$\zeta_{xy}(t) = 1$ for all $x\leq y\in P$.  The element $\mu := \zeta^{-1}\in\scrI(P)$
is called the {\bf M\"obius function}, and
the product $\chi := \mu\bar\zeta = \zeta^{-1}\bar\zeta$ is called the {\bf characteristic function} of $P$.
We then have $\chi^{-1} = \bar\zeta^{-1}\zeta = \bar\chi$, so $\chi$ is a $P$-kernel.
Proposition \ref{prop:R}(3) tells us that the associated left KLS-function is $\zeta$; this was observed by Stanley in \cite[Example 6.8]{Stanley-h}.
However, the associated right KLS-function $f$ can be much more interesting!  (In particular, $\chi$ is generally not alternating.)
For example, if $P$ is the lattice of flats of a matroid $M$ with the usual weak rank function, 
with minimum element 0 and maximum element 1, then $f_{01}(t)$
is the {\bf Kazhdan-Lusztig polynomial of \boldmath{$M$}} as defined in \cite{EPW}, and $Z_{01}(t)$
is the {\bf \boldmath{$Z$}-polynomial of \boldmath{$M$}} as defined in \cite{PXY}.
In general, the coefficients of $f_{xy}(t)$ can be expressed as alternating sums of multi-indexed Whitney numbers for the interval $[x,y]\subset P$;
see \cite[Corollary 6.5]{Brenti-twisted}, \cite[Theorem 5.1]{Wak}, and \cite[Theorem 3.3]{PXY} for three different formulations of this result.
\end{example}

\begin{example}\label{ex:Eulerian}
Let $P$ be any locally finite weakly ranked poset.
Define $\la\in\scrI(P)$ by the formula $\la_{xy}(t) = (t-1)^{r_{xy}}$ for all $x\leq y\in P$.
The weakly ranked poset $P$ is called {\bf locally Eulerian} if $\mu_{xy}(t) = (-1)^{r_{xy}}$ for all $x\leq y\in P$,
which is equivalent to the condition that $\la$ is a $P$-kernel \cite[Proposition 7.1]{Stanley-h}.
The poset of faces of a polytope, with weak rank function given by relative dimension
(where $\dim \emptyset = -1$), is Eulerian.  More generally, any fan is an Eulerian poset.

Let $\Delta$ be a polytope, let $P$ be the poset of faces of $\Delta$, and let $f$ and $g$ be the associated right and left KLS-functions.
Then $g_{\emptyset\Delta}(t)$ is called the {\bf \boldmath{$g$}-polynomial} of $\Delta$ \cite[Example 7.2]{Stanley-h}.
Since the dual polytope $\Delta^*$ has the property that its face poset is opposite to $P$,
and since $\lambda$ depends only on the weak rank function, Remark \ref{opposite} tells us that the right KLS-polynomial
$f_{\emptyset\Delta}(t)$ is equal to the $g$-polynomial of $\Delta^*$.  On the other hand, since $\la$ is clearly alternating,
Proposition \ref{antiK} tells us that $\hat g = f^{-1}$ and $\hat f = g^{-1}$ \cite[Corollary 8.3]{Stanley-h}.
\end{example}

\begin{remark}\label{BB}
For $P$ locally Eulerian,
Batyrev and Borisov define an element $B \in\prod_{x\leq y\in P} \Z[u,v]$ \cite[Definition 2.7]{Bat-Bor}.
Let $B'\in \scrI(P)$ be the function obtained from $B$ by setting $u=-t$ and $v=-1$.
The defining equation for $B$ transforms into the equation $B'\bar{\hat{f}} = f$.  Using the fact that $\hat f = g^{-1}$,
this means that $B' = f\bar g$.  Thus $B'$ is similar to $Z = \bar g f$, but it is not quite the same.
In particular, $B'$ need not be symmetric.
\end{remark}

\begin{example}\label{ex:hypertoric}
Let $M$ be a matroid with lattice of flats $L$.
Let $r\in\scrI(L)$ be the usual weak rank function, and let $\chi\in\scrI(L)$ be the characteristic function.
In this example, we will be interested in the weakly ranked poset $(L,2r)$, where $2r$ is 2 times the usual weak rank
function.

Define $\kappa\in\scrI(L,2r)$ by the following formula:
$$\kappa_{FH}(t) := (t-1)^{r_{FH}}\sum_{F\leq G \leq H}(-1)^{r_{FG}}\chi_{FG}(-1)\chi_{GH}(t).
$$
Define $\hbc\in\Ih(L,2r)$ by letting $$\hbc_{FG}(t) := (-t)^{r_{FG}}\chi_{FG}(1-t^{-1})$$ 
be the $h$-polynomial of the broken circuit complex
of $M_G^F$, where $M_G^F$ is the matroid on $G\smallsetminus F$ 
whose lattice of flats is isomorphic to $[F,G]\subset L$.

\begin{proposition}\label{hypertoric KL}
The function $\kappa$ is an $(L,2r)$-kernel, and $\hbc$ is its associated left KLS-function.
\end{proposition}

\begin{proof}
By Proposition \ref{prop:R}(3), it will suffice to show that $\overline{\hbc} = \hbc \kappa$.
We follow the argument in the proof of \cite[Theorem 4.3]{PW07}.  We will 
write $\mu_{FG}$ and $\delta_{FG}$ to denote the constant polynomials $\mu_{FG}(t)$ and $\delta_{FG}(t)$.  
For all $D\leq J$, we have
\begin{eqnarray*}
(\hbc\kappa)_{DJ}(t) &=& \sum_{D\leq F \leq J}\hbc_{DF}(t)\kappa_{FJ}(t)\\
&=& \sum_{D\leq F\leq H \leq J} (t-1)^{r_{FJ}}(-1)^{r_{FH}}\chi_{FH}(-1)\chi_{HJ}(t) 
(-t)^{r_{DF}}\chi_{DF}(1-t^{-1})\\
&=& \sum_{D\leq E\leq F\leq G\leq H \leq I\leq J}
(t-1)^{r_{FJ}}(-1)^{r_{FH}}\mu_{FG}(-1)^{r_{GH}}\mu_{HI} t^{r_{IJ}} (-t)^{r_{DF}}\mu_{DE}(1-t^{-1})^{r_{EF}}\\
&=& \sum_{D\leq E\leq F\leq G\leq H \leq I\leq J}
\mu_{DE}\mu_{FG}\mu_{HI}(-1)^{r_{DG}}t^{r_{DE}+ r_{IJ}}(t-1)^{r_{EJ}}\\
&=& \sum_{D\leq E\leq G \leq I\leq J}\mu_{DE}(-1)^{r_{DG}}t^{r_{DE}+ r_{IJ}}(t-1)^{r_{EJ}}
\sum_{E\leq F\leq G}\mu_{FG} \sum_{G\leq H \leq I}\mu_{HI}\\
&=& \sum_{D\leq E\leq G \leq I\leq J}\mu_{DE}(-1)^{r_{DG}}t^{r_{DE}+ r_{IJ}}(t-1)^{r_{EJ}}
\delta_{EG} \delta_{GI}\\
&=& \sum_{D\leq E\leq J}\mu_{DE}(-1)^{r_{DE}}t^{r_{DJ}}(t-1)^{r_{EJ}}\\
&=& (-t)^{r_{DJ}} \sum_{D\leq E\leq J}\mu_{DE}(1-t)^{r_{EJ}}\\
&=& t^{2r_{DJ}} (-t^{-1})^{r_{DJ}}\chi_{DJ}(1-t)\\
&=& t^{r_{DJ}} \hbc_{DJ}(t^{-1})\\
&=& \overline{\hbc}_{DJ}(t).
\end{eqnarray*}
This completes the proof.
\end{proof}
\end{example}

\excise{
\vspace{\baselineskip}
\begin{remark}
By Proposition \ref{prop:R}, any polynomial can be a KLS-polynomial, so it is not interesting to ask what properties KLS-polynomials have in general.
But it can be an interesting question for a particular class of weakly ranked posets and kernels, and the answers vary widely.
Classical Kazhdan-Lusztig polynomials (Example \ref{ex:coxeter}) always have constant term 1
and non-negative coefficients \cite[Theorem 1.2(1)]{EW14}, but otherwise they are arbitrary \cite{Polo}.
Similarly, Kazhdan-Lusztig coefficients of matroids (Example \ref{ex:char}) have 
constant term 1 \cite[Proposition 2.11]{EPW}
and are conjectured to have non-negative coefficients \cite[Conjecture 2.3]{EPW}, but they are also conjectured to
be real-rooted \cite[Conjecture 3.2]{kl-survey}.

In Section \ref{sec:geometry}, we will give a general framework for giving cohomological interpretations
of right KLS-polynomials for various classes of weakly ranked posets and kernels.  
In particular, this will provide a means
of proving that certain KLS-polynomials have constant term 1 and non-negative coefficients.
\end{remark}
}

\section{Geometry}\label{sec:geometry}
In this section we give a general geometric framework for interpreting right KLS-polynomials
in terms of the stalks of intersection cohomology sheaves on a stratified space.
Under some additional assumptions, we also give cohomological interpretations for the associated $Z$-polynomials.
Our primary reference for technical properties of intersection cohomology will be the book of Kiehl and Weissauer \cite{KW},
however, a reader who is learning this material for the first time might also benefit from the friendly
discussion in the book of Kirwan and Woolf \cite[Section 10.4]{KirwanWoolf}.

\subsection{The setup}\label{sec:setup}
Fix a finite field $\Fq$, an algebraic closure $\Fqb$, and a prime $\ell$ that does not divide $q$.
For any variety $Z$ over $\Fq$, let $\IC_Z$ denote the $\ell$-adic intersection cohomology sheaf on the variety $Z(\Fqb)$.
We adopt the convention of {\em not} shifting $\IC_Z$ to make it perverse.  
In particular, if $Z$ is smooth, then $\IC_Z$ is isomorphic to the constant sheaf in degree zero.

Suppose that we have a variety $Y$ over $\Fq$ and a stratification
$$Y = \bigsqcup_{x\in P} V_x.$$
By this we mean that each stratum $V_x$ is a smooth connected subvariety of $Y$ and the closure of each stratum
is itself a union of strata.
We define a partial order on $P$ by putting $x\leq y\iff V_x\subset\bar V_y$, and
a weak rank function by the formula $r_{xy} = \dim V_y - \dim V_x$.
Fix a point $e_x\in V_x$ for each $x\in P$.

Next, suppose that we have a stratification preserving $\gm$-action $\rho_x:\gm\to\Aut(Y)$ for each $x\in P$ 
and an affine $\gm$-subvariety $C_x\subset Y$ with the following properties:
\begin{itemize}
\item $C_x$ is a weighted affine cone with respect to $\rho_x$ with cone point $e_x$.
In other words, the $\Z$-grading on the affine coordinate ring $\Fq[C_x]$ induced by $\rho_x$ is non-negative
and the vanishing locus of the ideal of positively graded elements is $\{e_x\}$.
\item For all $x,y\in P$, let 
$$U_{xy} := C_x \cap V_y\and X_{xy} := C_x \cap \bar V_y.$$
We require that
the restriction of $\IC_{\bar V_y}$ to $C_x(\Fqb)$ is isomorphic to
$\IC_{X_{xy}}$.
\end{itemize}

Note that the variety $X_{xy}$ is a closed $\gm$-equivariant subvariety of $C_x$, therefore it is either empty or a weighted
affine cone with cone point $e_x$.  We have $$e_x\in X_{xy} \iff e_x\in \bar V_y \iff x\leq y,$$
so $X_{xy}$ is nonempty if and only if $x\leq y$.

\begin{lemma}\label{X-strat}
For all $x\leq z$, we have $\displaystyle X_{xz} = \bigsqcup_{x\leq y\leq z} U_{xy}$.
\end{lemma}

\begin{proof}
We have $X_{xz} = C_x\cap \bar V_z = C_x \cap \bigsqcup_{y\leq z} V_y = \bigsqcup_{y\leq z} U_{xy}.$
If $x$ is not less than or equal to $Y$, then $X_{xy}$ is empty, thus so is $U_{xy}$.
\end{proof}

The condition on restrictions of IC sheaves is somewhat daunting.
In each of our families of examples, we will check this condition by means of a group
action, using the following lemma.

\begin{lemma}\label{orbit}
Suppose that $Y$ is equipped with an action of an algebraic group $G$ preserving the stratification.
Suppose in addition that, for each $x\in P$, there exists a subgroup $G_x\subset G$ such that the composition
$$\varphi_x: G_x\times C_x\hookrightarrow G\times Y \to Y$$ is an open immersion.
Then for all $x\leq y\in P$, the restriction
of $\IC_{\bar V_y}$ to $C_x(\Fqb)$ is isomorphic to $\IC_{X_{xy}}$.
\end{lemma}

\begin{proof}
Since $\varphi_x$ is an open immersion, we have
$\varphi_x^{-1}\IC_{\bar V_y} \cong \IC_{\varphi_x^{-1}(\bar V_y)}$ as sheaves on $G_x(\Fqb)\times C_x(\Fqb)$ for all $x,y\in P$.
Since the action of $G$ on $Y$ preserves the stratification, we have $$\varphi_x^{-1}(\bar V_y) = G_x\times (C_x\cap \bar V_y) = G_x \times X_{xy},$$
so $\varphi_x^{-1}\IC_{\bar V_y} \cong \IC_{G_x}\boxtimes\IC_{X_{xy}}$.
Since $G_x$ is smooth, $\IC_{G_x}$ is the constant sheaf on $G_x$.  Thus, if we further restrict
to $C_x(\Fqb) \cong \{\operatorname{id}_{G_x}\}\times C_x(\Fqb)$, we obtain $\IC_{X_{xy}}$.
\end{proof}

\begin{remark}
In some of our examples (Sections \ref{sec:gr} and \ref{sec:arrangement}), the $\gm$-action $\rho_x$ will not actually depend on $x$.
In other examples (Sections \ref{sec:flag} and \ref{sec:toric}), it will depend on $x$.
\end{remark}

\subsection{Intersection cohomology}
We will write $\IH^*(Z)$ and $\IH^*_c(Z)$ to denote the ordinary and compactly supported cohomology of $\IC_Z$.
Given a point $p\in Z$, we will write $\IH_p^*(Z)$ to denote the cohomology of the stalk of $\IC_Z$ at $p$.
Each of these graded $\Ql$-vector spaces has a natural Frobenius automorphism induced by the Frobenius automorphism of $Z$.
We will be interested in the vector spaces $\IH^*_{xy} := \IH^*_{e_x}(\bar V_y)$ for all $x\leq y$.

\begin{lemma}\label{slice}
If $x\leq y\leq z$ and $u\in U_{xy}$, then $\IH^*_{yz} \cong \IH^*_u(X_{xz})$.
\end{lemma}

\begin{proof}
Since $u$ and $e_y$ lie in the same connected stratum of $\bar V_z$, 
we have an isomorphism of stalks $\IC_{\bar V_z,e_y}\cong \IC_{\bar V_z,u}$.
Since the restriction of $\IC_{\bar V_z}$ to $C_x(\Fqb)$ is isomorphic to $\IC_{X_{xz}}$,
we have an isomorphism of stalks
$\IC_{\bar V_z,u} \cong \IC_{X_{xz},u}$.
Putting these two stalk isomorphisms together, we have
$$\IH^*_{yz} = \IH^*_{e_y}(\bar V_z)
= \H^*(\IC_{\bar V_z,e_y})
\cong \H^*(\IC_{\bar V_z,u})
\cong \H^*(\IC_{X_{xz},u})
= \IH^*_u(X_{xz}).$$
This completes the proof.
\end{proof}

\begin{lemma}\label{cone point}
For all $y\leq z$, $\IH^*_{yz} \cong \IH^*(X_{yz})$.
\end{lemma}

\begin{proof}
If we apply Lemma \ref{slice} with $x=y$, we find that $\IH^*_{yz} \cong \IH^*_{e_y}(X_{yz})$.
Since $X_{yz}$ is a weighted affine cone with cone point $e_y$, the cohomology of the stalk of the IC sheaf at $e_y$
coincides with the global intersection cohomology \cite[Lemma 4.5(a)]{KL80}.
\end{proof}

We call an intersection cohomology group {\bf chaste} if it vanishes in odd degrees and the Frobenius automorphism
acts on the degree $2i$ part by multiplication by $q^{i}$ \cite[Section 3.3]{EPW}.  (This is much stronger than being {\bf pure},
which is a statement about the absolute values of the eigenvalues of the Frobenius automorphism.)

\subsection{Right KLS-polynomials}
Define $f\in \scrI(P)$ by putting $$f_{xy}(t) := \sum_{i\geq 0}t^i\dim\IH^{2i}_{xy}$$ for all $x\leq y$. 
We observe that $f\in\Ih(P)$ by Lemma \ref{cone point} and \cite[Proposition 3.4]{EPW}.

\begin{theorem}\label{main}
Suppose that we have an element $\kappa\in\scrI(P)$ 
such that, for all $x\leq y$ and all positive integers $s$, $$\kappa_{xy}(q^s) = |U_{xy}(\mathbb{F}_{q^s})|.$$
Then $\IH^*_{xz}$ is chaste for all $x\leq z$, $\kappa$ is a $P$-kernel, and $f$ is the associated right KLS-function. 
\end{theorem}

\begin{remark}
The first time that you read the proof of Theorem \ref{main}, 
it is helpful to pretend that we already know that $\IH^*_{xz}$ is chaste for all $x\leq z$.
In this case, the proof simplifies to a straightforward application of Poincar\'e duality and the Lefschetz formula, 
along with Lemmas \ref{X-strat}, \ref{slice}, and \ref{cone point}.  The actual proof 
as it appears is made significantly more subtle by the need
to fold the chastity statement into the induction.
\end{remark}

\begin{proofmain}
%
We begin with an inductive proof of chastity.  It is clear that $\IH^*_{xx}$ is chaste for all $x\in P$.  Now consider a pair of elements
$x<z$, and assume that $\IH^*_{yz}$ is chaste for all $x<y\leq z$.
%
%
Let $s$ be any positive integer.  Applying the Lefschetz formula \cite[III.12.1(4)]{KW}, 
along with Lemmas \ref{X-strat} and \ref{slice}, we find that
\begin{eqnarray*}\sum_{i\geq 0}(-1)^i \operatorname{tr}\!\Big(\Fr^s\curvearrowright \IH_c^i(X_{xz})\Big)
&=& \sum_{u\in X_{xz}(\Fqs)}\sum_{i\geq 0}(-1)^i \operatorname{tr}\!\Big(\Fr^s\curvearrowright\IH^i_u(X_{xz})\Big)\\
&=& \sum_{x\leq y\leq z}\sum_{u\in U_{xy}(\Fqs)}\sum_{i\geq 0}(-1)^i \operatorname{tr}\!\Big(\Fr^s\curvearrowright\IH^i_u(X_{xz})\Big)\\
&=& \sum_{x\leq y\leq z}\sum_{u\in U_{xy}(\Fqs)}\sum_{i\geq 0}(-1)^i \operatorname{tr}\!\Big(\Fr^s\curvearrowright\IH^i_{yz}\Big)\\
&=& \sum_{x\leq y\leq z} \kappa_{xy}(q^s)\sum_{i\geq 0}(-1)^i \operatorname{tr}\!\Big(\Fr^s\curvearrowright\IH^i_{yz}\Big).
\end{eqnarray*}
By Poincar\'e duality \cite[II.7.3]{KW}, we have
\begin{equation*}\label{pd}
\operatorname{tr}\!\Big(\Fr^s\curvearrowright \IH_c^i(X_{xz})\Big)
= q^{s\rxz}\operatorname{tr}\!\Big(\Fr^{-s}\curvearrowright \IH^{2\rxz - i}(X_{xz})\Big).
\end{equation*}
By our inductive hypothesis, we have
$\sum_{i\geq 0}(-1)^i \operatorname{tr}\!\Big(\Fr^s\curvearrowright\IH^*_{yz}\Big) = f_{yz}(q^s)$
for all $x<y\leq z$.  Moving the $x=y$ term from the right hand side to the left hand side, the Lefschetz formula becomes
\begin{equation}\label{lef}
\sum_{i\geq 0}(-1)^i \left(q^{s\rxz} \operatorname{tr}\!\Big(\Fr^{-s}\curvearrowright \IH^{2\rxz-i}_{xz}\Big) - \operatorname{tr}\!\Big(\Fr^{s}\curvearrowright \IH^i_{xz}\Big)\right)
= \sum_{x<y\leq z}\kappa_{xy}(q^s)f_{yz}(q^s).
\end{equation}

We now follow the proof of \cite[Theorem 3.7]{EPW}.  Let $b_i = \dim\IH^i_{xz}$.
Let $(\a_{i,1},\ldots,\a_{i,b_i})\in\overline{\Q}_\ell^{b_i}$ be the eigenvalues of the Frobenius action on $\IH^i_{xz}$ 
(with multiplicity, in any order).  Then Equation \eqref{lef} becomes
$$\sum_{i\geq 0}(-1)^i\sum_{j=1}^{b_i} \Big( (q^{r_{xz}}/\a_{i,j})^s - \a_{i,j}^s \Big)
= \sum_{x<y\leq z}\kappa_{xy}(q^s)f_{yz}(q^s).$$
By Lemma \ref{cone point} and \cite[Proposition 3.4]{EPW}, $\IH^i_{xz} = 0$ for $i\geq r_{xz}$, and for
any $i<r_{xz}/2$, $\a_{i,j}$ has absolute value $q^{i/2} < q^{r_{xz}/2}$.  It follows that
$q^{r_{xz}}/\a_{i,j}$ has absolute value $q^{r_{xz}-i/2} > q^{r_{xz}/2}$, and therefore that 
the numbers that appear with positive sign on the left-hand side of Equation \eqref{lef}
are pairwise disjoint from the numbers that appear with negative sign.
Since the right-hand side is a sum of integer powers of $q^s$ with integer coefficients, \cite[Lemma 3.6]{EPW}
tells us that each $\a_{i,j}$ must also be an integer power of $q$.  This is only possible if $b_i = 0$ for odd $i$
and $\a_{i,j} = q^{i/2}$ for even $i$, thus $\IH^*_{xz}$ is chaste.

Now that we have established chastity, Equation \eqref{lef} becomes
$$q^{s\rxz}f_{xz}(q^{-s}) - f_{xz}(q^s) = \sum_{x<y\leq z}\kappa_{xy}(q^s)f_{yz}(q^s),$$
or equivalently $$\bar f_{xz}(q^s) = q^{s\rxz}f_{xz}(q^{-s}) = \sum_{x\leq y\leq z}\kappa_{xy}(q^s)f_{yz}(q^s) = (\kappa f)_{xz}(q^s).$$
Since this holds for all positive $s$, it must also hold with $q^s$ replaced by the formal variable $t$,
thus $\bar f = \kappa f$.  The fact that $\kappa$ is a $P$-kernel with $f$ as its associated right KLS-function
now follows follow from Proposition \ref{prop:R}(2).
\end{proofmain}

\vspace{\baselineskip}
The same idea used in the proof of Theorem \ref{main} can be used to obtain the following converse.

\begin{theorem}\label{converse}
Suppose that $\IH^*_{xz}$ is chaste for all $x\leq z$, and let $\kappa := \bar f f^{-1}$.
Then for all $s>0$ and $x\leq z$,  
$$\kappa_{xz}(q^s) = |U_{xz}(\mathbb{F}_{q^s})|.$$
\end{theorem}

\begin{proof}
We proceed by induction.  When $x=z$, we have $\kappa_{xz}(t) = 1$ and $U_{xz} = \{e_x\}$,
so the statement is clear.  Now assume that $\kappa_{xy}(q^s) = |U_{xy}(\mathbb{F}_{q^s})|$
for all $x\leq y<z$.  By Poincar\'e duality the Lefschetz formula, we have
$$\bar f_{xz}(q^s) = \sum_{x\leq y\leq z}  |U_{xy}(\mathbb{F}_{q^s})| f_{yz}(q^s)
=  |U_{xz}(\mathbb{F}_{q^s})| + \sum_{x\leq y< z}  \kappa(q^s) f_{yz}(q^s).$$
By the definition of $\kappa$, we have
$$\bar f_{xz}(q^s) = \sum_{x\leq y\leq z}  \kappa(q^s) f_{yz}(q^s).$$
Comparing these two equations, we find that $|U_{xz}(\mathbb{F}_{q^s})| = \kappa(q^s)$.
\end{proof}

\begin{remark}\label{it's okay}
In Section \ref{sec:gr}, we will apply Theorem \ref{converse} when $Y$ is the affine Grassmannian.  
Then $Y$ is an ind-scheme rather than a variety,
but each $\bar V_x$ is an honest variety, and the proof goes through without modification.
\end{remark}

\subsection{\boldmath{$Z$}-polynomials}\label{sec:Z-geom}
In this section we will explain how to give a cohomological interpretation of $Z$-polynomials under certain more restrictive hypotheses.
Specifically, we will assume that $\IH^*_{xy}$ is chaste for all $x\leq y$, let $\kappa := \bar f f^{-1}$, and let $g$ be the {\bf left}
KLS-function associated with $\kappa$.  We will also assume that there is a minimal element $0\in P$ and
a function $h\in\Ih(P)$ such that $\bar h_{0x}(q^s) = |V_x(\Fqs)|$ for all $x\in P$ and $s>0$.  Finally, we will assume that $\bar V_y$
is proper for all $y\in P$.

\begin{theorem}\label{coh-Z0}
Suppose that all of the above hypotheses are satisfied.
Then for all $y\in P$, we have $g_{0y}(t) = h_{0y}(t)$, $\IH^*(\bar V_y)$ is chaste, and
$$\sum_{i\geq 0} t^i \dim \IH^{2i}(\bar V_y) = Z_{0y}(t).$$
\end{theorem}

\begin{proof}
Following the proof of Theorem \ref{main}, we apply the Lefschetz formula to obtain
\begin{eqnarray*}\sum_{i\geq 0}(-1)^i \operatorname{tr}\!\Big(\Fr^s\curvearrowright \IH_c^i(\bar V_y)\Big)
&=& \sum_{v\in \bar V_y(\Fqs)}\sum_{i\geq 0}(-1)^i \operatorname{tr}\!\Big(\Fr^s\curvearrowright\IH^*_v(\bar V_y)\Big)\\
&=& \sum_{x\leq y}\sum_{v\in V_{x}(\Fqs)}\sum_{i\geq 0}(-1)^i \operatorname{tr}\!\Big(\Fr^s\curvearrowright\IH^*_v(\bar V_y)\Big)\\
&=& \sum_{x\leq y}\bar h_{0x}(q^s)\sum_{i\geq 0}(-1)^i \operatorname{tr}\!\Big(\Fr^s\curvearrowright\IH^*_{xy}\Big)\\
&=& \sum_{x\leq y}\bar h_{0x}(q^s) f_{xy}(q^s)\\
&=& (\bar h f)_{0y}(q^s).
\end{eqnarray*}
Since $\bar V_y$ is proper, compactly supported intersection cohomology coincides with ordinary intersection cohomology.
Poincar\'e duality then tells us that 
$(\bar h f)_{0y}(q^s) = (h\bar f)_{0y}(q^s)$.  Since this is true for all $s$, we must have
$(\bar h f)_{0y}(t) = (h\bar f)_{0y}(t)$.  By Proposition \ref{weak g}, we may conclude that $h_{0y}(t) = g_{0y}(t)$ for all $y\in P$, and therefore that
\begin{equation}\label{lef2}\sum_{i\geq 0}(-1)^i \operatorname{tr}\!\Big(\Fr^s\curvearrowright \IH^i(\bar V_y)\Big)=
Z_{0y}(q^s).\end{equation}

Let $b_i = \dim\IH^i(\bar V_y)$.
Let $(\a_{i,1},\ldots,\a_{i,b_i})\in\overline{\Q}_\ell^{b_i}$ be the eigenvalues of the Frobenius action on $\IH^i(\bar V_y)$ 
(with multiplicity, in any order).  Then Equation \eqref{lef2} becomes
$$\sum_{i\geq 0}(-1)^i\sum_{j=1}^{b_i} \a_{i,j}^s
= Z_{0y}(q^s).$$
By Deligne's theorem \cite[Theorems 3.1.5 and 3.1.6]{dCM}, each $\a_{i,j}$ has absolute value $q^{i/2}$.
Since the right-hand side is a sum of integer powers of $q^s$ with integer coefficients, \cite[Lemma 3.6]{EPW}
tells us that each $\a_{i,j}$ must also be an integer power of $q$.  This is only possible if $b_i = 0$ for odd $i$
and $\a_{i,j} = q^{i/2}$ for even $i$.
This proves that $\IH^i(\bar V_y)$ is chaste, and Equation \eqref{lef2} becomes
$$\sum_{i\geq 0} q^{is} \dim \IH^{2i}(\bar V_y) = Z_{0y}(q^s).$$
Since this holds for all positive $s$, it must also hold with $q^s$ replaced by the formal variable $t$.
\end{proof}

\begin{remark}
We will apply Theorem \ref{coh-Z0} in the case where $Y$ is a flag variety (Section \ref{sec:flag}), 
an affine Grassmannian (Section \ref{sec:gr}), or
the Schubert variety of a hyperplane arrangement (Section \ref{sec:arrangement}).
In the first and third cases, we will be able to make an even stronger statement, namely
that $$\sum_{i\geq 0} t^i \dim \IH^{2i}(\bar C_x \cap \bar V_y) = Z_{xy}(t)$$
(Theorems \ref{richardson} and \ref{Z-arr}).
However, this seems to be true for different reasons in the two cases, and we are unable to find a 
unified proof; see Remark \ref{different} for further discussion.
\end{remark}

\subsection{Category \boldmath{$\cO$}}\label{sec:O}
In this section we assume that the hypotheses of Theorem \ref{main} are satisfied,
and we make the additional assumption that each stratum $V_x$ is isomorphic to an affine space.
Though this is a very restrictive assumption, it is satisfied by two of our main families of examples
(Sections \ref{sec:flag} and \ref{sec:arrangement}).

For each $x\in P$, let $\cL_x := \IC_{\bar V_x}[\dim V_x]$, and let $\cO$ denote the Serre subcategory of $\Ql$-perverse sheaves on $Y(\Fqb)$ generated
by $\{\cL_x\mid x\in P\}$.  Let $\iota_x:V_x\to Y$ be the inclusion, and define
$$\cM_x := (\iota_x)_!{\Ql}_{V_x}[\dim V_x] \and \cN_x := (\iota_x)_*{\Ql}_{V_x}[\dim V_x].$$
Then $\cO$ is a highest weight category in with simple objects $\{\cL_x\}$, standard objects $\{\cM_x\}$, and costandard objects $\{\cN_x\}$
\cite[Lemmas 4.4.5 and 4.4.6]{BGS96}.
For all $x\leq y\in P$, we have $\Ext^j_\cO(\cM_x,\cL_y) = 0$ unless $j+r_{xy}$ is even, and
\begin{equation}\label{catO1}
f_{xy}(t) = \sum_{i\geq 0}t^i \dim \Ext_\cO^{r_{xy}-2i}(\cM_x,\cL_y).\end{equation}
Motivated by the examples in Section \ref{sec:flag},
Beilinson, Ginzburg, and Soergel prove that the category $\cO$ admits a grading, and the graded lift $\tilde\cO$ of $\cO$ is Koszul \cite[Theorem 4.4.4]{BGS96}.
The Grothendieck group of $\tilde \cO$ is a module over $\Z[t,t^{-1}]$ whose specialization at $t=1$ is canonically isomorphic to the Grothendieck group of $\cO$.
If $\tilde \cL_x$ and $\tilde \cN_x$ are the natural lifts to $\tilde\cO$ of $\cL_x$ and $\cN_x$, then 
we have \cite[Equation (3.0.6)]{CPS-abstract} \begin{equation}\label{catO2}[\tilde\cL_y] = \sum_{x\leq y} \bar f_{xy}(t^2) [\tilde\cN_x].\end{equation}
More generally, Cline, Parshall, and Scott study abstract frameworks for obtaining categorical (rather than cohomological) 
interpretations of Kazhdan-Lusztig-Stanley polynomials
\cite{CPS-abstract,CPS-graded}.

\section{Examples}\label{sec:examples}
In this section we apply the results of Section \ref{sec:geometry} to a number of different families of examples.

\subsection{Flag varieties}\label{sec:flag}
%
Let $G$ be a split reductive algebraic group over $\Fq$.  Let $B,B^*\subset G$ be Borel subgroups with the property 
that $T := B\cap B^*$ is a maximal torus.  Let $W := N(T)/T$ be the Weyl group.  
Let $Y := G/B$ be the {\bf flag variety} of $G$.
For all $w\in W$, let
$$V_w := \{gB\mid g\in BwB\}\and C_w:= \{gB\mid g\in B^*wB\}.$$
Let $e_w := wB$ be the unique element of $C_w\cap V_w$.
The variety $V_w$ is called a {\bf Schubert cell}, and $C_w$ is called an {\bf opposite Schubert cell}.
The flag variety is stratified by Schubert cells, and the induced partial order on $W$ is called the {\bf Bruhat order}.

The existence of the homomorphism $\rho_w:\gm\to T\subset G$ exhibiting $C_w$ as a weighted affine cone
is proved in \cite[Lemma A.6]{KL79} (see alternatively \cite[Section 1.5]{KL80}).
Let $N\subset B$ and $N^*\subset B^*$ be the unipotent radicals, and for
each $w\in W$, let $N_w := N \,\cap\, w N^* w^{-1}$.  Then $N_w$ acts freely and transitively on $V_w$ and 
the action map $N_w \times C_w\to Y$ is an open immersion \cite[Section 1.4]{KL80}.  In particular, Lemma \ref{orbit} applies.

For all $v\leq w$, let $U_{vw} := C_v\cap V_w$.
Kazhdan and Lusztig show that
$R_{vw}(q) = |U_{vw}(\Fq)|$ in \cite[Lemma A.4]{KL79} (see alternatively \cite[Section 4.6]{KL80}), where $R$
is the $W$-kernel of Example \ref{ex:coxeter}.  We therefore obtain the
following corollary to Theorem \ref{main}, which first appeared in \cite[Theorem 3.3]{KL80}.

\begin{corollary}\label{KL-right}
Let $f\in\Ih(W)$ be the right KLS-function associated with $R\in\scrI(W)$.  
For all $v\leq w\in W$, $\IH^{*}_{e_v}(\bar V_w)$ is chaste and $$f_{vw}(t) = \sum_{i\geq 0}t^i \dim \IH^{2i}_{e_v}(\bar V_w).$$
\end{corollary}

For each $w\in W$, the Schubert cell $V_w\cong N_w$ is isomorphic to an affine space of dimension $\ell(w) = r_{ew}$ 
(where $e\in W$ is the identity element) \cite[Section 1.3]{KL80}.  We therefore obtain the following corollary to Theorem \ref{coh-Z0}, which
originally appeared in \cite[Corollary 4.8]{KL80}.

\begin{corollary}\label{kl-Z0}
For all $w\in W$, $g_{ew}(t) = 1$, $\IH^{*}(\bar V_w)$ is chaste, and $$Z_{ew}(t) = \sum_{i\geq 0}t^i \dim \IH^{2i}(\bar V_w).$$
\end{corollary}

Next, we use features unique to this particular class of examples to describe $Z_{vw}(t)$ for arbitrary $v\leq w\in W$.  Let $\tilde w_0 \in N(T)\subset G$ be a lift of $w_0\in W$.  Then we have $\tilde w_0V_w = C_{w_0w}$
and $\tilde w_0C_w = V_{w_0w}$.  In particular, this implies that $\IH^*_{e_w}(\bar C_v)$ is chaste
for all $v\leq w$, and 
\begin{equation}\label{opp}
g_{vw}(t) = f_{(w_0w)(w_0v)}(t) = \sum_{i\geq 0}t^i \dim \IH^{2i}_{e_w}(\bar C_v)\end{equation}
for all $v\leq w\in W$.  Consider the {\bf Richardson variety} $\bar C_v\cap \bar V_w$.

\begin{theorem}\label{richardson}
For all $x\leq w\in W$, $\IH^*(\bar C_x \cap \bar V_w)$ is chaste and $$Z_{xw}(t) = \sum_{i\geq 0}t^i \dim \IH^{2i}(\bar C_x \cap \bar V_w).$$
\end{theorem}

\begin{proof}
Knutson, Woo, and Yong \cite[Section 3.1]{KWY} prove that, 
for all $x\leq y\leq z\leq w\in W$ and $u\in U_{yz}$, we have 
\begin{equation}\label{KWY}\IH^*_u(\bar C_x \cap \bar V_w)\cong \IH^*_u(\bar C_x)\otimes \IH^*_u(\bar V_w)
\cong \IH^*_{e_y}(\bar C_x)\otimes \IH^*_{e_z}(\bar V_w),\end{equation}
and therefore $$\sum_{i\geq 0}(-1)^i \operatorname{tr}\!\Big(\Fr^s\curvearrowright \IH^*_u(\bar C_x \cap \bar V_w)\Big) = g_{xy}(q^s)f_{zw}(q^s).$$
Applying the Lefschetz formula, we have 
\begin{eqnarray*}\sum_{i\geq 0}(-1)^i \operatorname{tr}\!\Big(\Fr^s\curvearrowright \IH^i(\bar C_x \cap \bar V_w)\Big) &=&
\sum_{u\in \bar C_w(\Fqs) \cap \bar V_x(\Fqs)} \sum_{i\geq 0}(-1)^i \operatorname{tr}\!\Big(\Fr^s\curvearrowright \IH^*_u(\bar C_w \cap \bar V_x)\Big)\\
&=& \sum_{x\leq y\leq z\leq w} \sum_{u\in U_{yz}(\Fqs)} \sum_{i\geq 0}(-1)^i \operatorname{tr}\!\Big(\Fr^s\curvearrowright \IH^*_u(\bar C_w \cap \bar V_x)\Big)\\
&=& \sum_{x\leq y\leq z\leq w} g_{xy}(q^s)R_{yz}(q^s)f_{zw}(q^s)\\
&=& (gRf)_{xz}(q^s)\\
&=& Z_{xz}(q^s).
\end{eqnarray*}
By the same argument employed in the proofs of Theorems \ref{main} and \ref{coh-Z0}, this implies that 
$\IH^*(\bar C_x \cap \bar V_w)$ is chaste and $\displaystyle Z_{xw}(t) = \sum_{i\geq 0}t^i \dim \IH^{2i}(\bar C_x \cap \bar V_w).$
\end{proof}

\begin{remark}
By the observation at the end of Example \ref{ex:coxeter}, we have $g_{xy}(t) = f_{(w_0y)(w_0x)}(t)$, and therefore
$$Z_{xw}(t) = \sum_{x\leq y\leq w} \bar g_{xy}(t) f_{yw}(t) = \sum_{x\leq y\leq w} \bar f_{(w_0y)(w_0x)}(t) f_{yw}(t).$$
Thus it is possible to express the intersection cohomology Poincar\'e polynomial 
of a Richardson variety as a sum of products of classical Kazhdan-Lusztig
polynomials (one of which is barred).  If $x=e$ (as in Corollary \ref{coh-Z0}), then $f_{(w_0y)(w_0x)}(t) = 1$,
so $\bar f_{(w_0y)(w_0x)}(t) = t^{r_{xy}}$
and we obtain the well-known formula for the intersection cohomology Poincar\'e polynomial of $\bar V_w$. 
\end{remark}

\begin{remark}
Since each $V_w$ is isomorphic to an affine space, the results of Section \ref{sec:O} apply.
The category $\cO$ is equivalent to a regular block of the Bernstein-Gelfand-Gelfand category $\cO$
for the Lie algebra $\Lie(G)$.
\end{remark}

\subsection{The affine Grassmannian}\label{sec:gr}
Let $G$ be a split reductive group over $\Fq$ with maximal torus $T\subset G$, 
and let $G^\vee$ be the Langlands dual group.
Let $\Lambda$ denote the lattice of coweights of $G$ (equivalently weights of $G^\vee$),
and let $\Lambda^\vee$ be the dual lattice.
Let $2\rho^\vee\in \Lambda^\vee$ be the sum of the positive roots of $G$.
Let $\Lambda^+\subset \Lambda$ be the set of dominant weights of $G^\vee$, equipped with the partial
order $\mu\leq \la$ if and only if $\la-\mu$ is a sum of positive roots.
This makes $\Lambda^+$ into a locally finite poset, and we endow it with the weak rank function
$$r_{\mu\la} := \langle\la-\mu,2\rho^\vee\rangle.$$

Let $Y:= G((s))/G[[s]]$ be the affine Grassmannian for $G$.  
We have a natural bijection
between $\Lambda$ and $T((s))/T[[s]]$.  For any $\la\in \Lambda^+\subset \Lambda \cong T((s))/T[[s]]$,
let $\tilde\lambda$ be a lift of $\lambda$ to $T((s))\subset G((s))$, and let $e_\la$ be the image of $\tilde\lambda$
in $Y$, which is independent of the choice of lift.  Let
$$V_\la := \Gr^\la := G[[s]]\cdot e_\la\subset Y.$$
This subvariety is smooth of dimension $\langle\la,2\rho^\vee\rangle$, 
and we have a stratification
$$Y = \bigsqcup_{\la\in \Lambda^+} V_\la$$ inducing the given weakly ranked poset structure on $\Lambda^+$;
see, for example, \cite[Lemma 2.2]{BF}.

For any $\mu\leq\la$, let $L(\la)_\mu$ denote the $\mu$ weight space of the irreducible representation
of $G^\vee$ with highest weight $\la$.  The vector space $L(\la)_\mu$ is filtered by the annihilators
of powers of a regular nilpotent element of $\Lie(G^\vee)$, and it follows from the work of Lusztig and Brylinski
that the intersection cohomology group $\IH^*_{\mu\la}$ is canonically isomorphic 
as a graded vector space to the associated graded of this filtration \cite[Theorem 2.5]{BF}.
Moreover, it is chaste \cite[Theorem 2.0.1]{Haines}.
(The vanishing of $\IH^*_{\mu\la}$ in odd degree is originally due to Lusztig \cite[Section 11]{Lusztig-singularities},
and the discussion there makes it clear that he was aware that it is chaste, but the full statement
of chastity does not appear explicitly.)  The polynomial $$f_{\mu\la}(t) := \sum_{i\geq 0}t^i \dim \IH^{2i}_{\mu\la}$$
goes by many names, including {\bf spherical affine Kazhdan-Lusztig polynomial}, {\bf Kostka-Foulkes polynomial},
and the {\bf \boldmath{$t$}-character of $L(\la)_\mu$}.  For a detailed discussion of various combinatorial interpretations,
see \cite[Theorem 3.17]{NelsenRam}.

For any $\mu\in \Lambda^+$, let $$C_\mu := \mathcal{W}_\mu := s^{-1}G[s^{-1}]\cdot e_\la\subset Y.$$
The space $C_\mu$ is infinite dimensional, but, as in Section \ref{sec:setup}, we will only be interested in the
finite dimensional varieties $$U_{\mu\la} := C_\mu\cap V_\la\and X_{\mu\la} := C_\mu\cap \bar V_\la.$$
These varieties satisfy the two conditions of Section \ref{sec:setup}; that is, each $X_{\mu\la}$ is a weighted affine cone
with respect to loop rotation, and the restriction of $\IC_{\bar V_\la}$ to $X_{\mu\la}(\Fqb)$ is isomorphic to
$\IC_{X_{\mu\la}}$ \cite[Lemma 2.9]{BF} (see also \cite[Proposition 2.3.9]{Zhu}).  
In particular, we have the following corollary to Theorem \ref{converse}.

\begin{corollary}
Let $\kappa := \bar f f^{-1}\in \scrI(\Lambda^+)$.  Then for all $s>0$ and $\mu\leq\la\in \Lambda^+$,
$\kappa(q^s) = |U_{\mu\la}(\Fqs)|.$
\end{corollary}

\begin{remark}
We have used the fact that $\IH^*_{\mu\la}$ is chaste to determine that $|U_{\mu\la}(\Fqs)|$ is a polynomial in $q^s$,
and that one can obtain a formula for this polynomial by inverting the matrix of spherical affine Kazhdan-Lusztig polynomials.
It would be interesting to prove directly that $U_{\mu\la}(\Fqs)$ is a polynomial in $q^s$, both because it would
be nice to have an explicit formula for this polynomial, and because it would provide a new proof of chastity.
\end{remark}

We now say something about the geometry of the varieties $V_\la$ and $Z$-polynomials.
Let $g,Z\in\scrI(\Lambda^+)$ be the left KLS-polynomial and the $Z$-polynomial associated with $\kappa$.
For each $\la\in \Lambda^+$, let $P_\la\subset G$ be the parabolic subgroup generated by the root subgroups for roots that
pair non-positively with $\la$.  In particular, $P_0 = G$, and $P_\la = B$ for generic $\la$.  Let $W_\la\subset W$
be the stabilizer of $\la$ in the Weyl group.
Then $V_\la$ is an affine bundle over $G/P_\la$ \cite[Section 2]{Zhu}, which allows us to compute \cite[Equation (8.10) and Section 11]{Lusztig-singularities}
$$|V_{\la}(\Fqs)| = q^{\langle \la, 2\rho^\vee\rangle - \nu_0 + \nu_\la}\frac{\sum_{w\in W} q^{\ell(w)}}{\sum_{w\in W_\la} q^{\ell(w)}}.$$
Then we have the following corollary to Theorem \ref{coh-Z0}.

\begin{corollary}\label{Z-affine}
For all $\la\in\Lambda^+$, 
we have $$g_{0\la}(t) = t^{\nu_0-\nu_\la} \frac{\sum_{w\in W} t^{-\ell(w)}}{\sum_{w\in W_\la} t^{-\ell(w)}},$$
$\IH^*(\bar V_\la)$ is chaste, and
$$Z_{0\la}(t) = \sum_{i\geq 0} t^i\dim \IH^{2i}(\bar V_\la).$$
\end{corollary}

\begin{remark}
Lusztig \cite[Equation (8.10)]{Lusztig-singularities} tells us that
$$Z_{0\la}(t) = \prod_{\a\in\Delta_+}\frac{t^{\langle \la +\rho,\a^\vee\rangle}-1}{t^{\langle \la,\a^\vee\rangle}-1},$$
where $\Delta_+\subset \Lambda^\vee$ is the set of positive roots for $G$.  Since the geometric Satake isomorphism identifies
$\IH^*(\bar V_\la)$ with $L(\la)$, we also obtain the equation $Z_{0\la}(1) = \dim L(\la)$.
\end{remark}

\subsection{Hyperplane arrangements}\label{sec:arrangement}
%
Let $V$ be a vector space over $\Fq$, and let $\cA = \{H_i\mid i\in \cI\}$ be an essential central
arrangement of hyperplanes in $V$.  
For each $i\in \cI$, let $\Lambda_i := V/H_i$, and let
$\bP_i := \bP(\Lambda_i \oplus \Fq) = \Lambda_i \cup\{\infty\}$ be the projective completion of $\Lambda_i$.
Let $\Lambda := \bigoplus_{i\in\cI}\Lambda_i$ and $\bP := \prod_{i\in \cI}\bP_i.$
We have a natural linear embedding $V\subset \Lambda\subset\bP$, and we define
$$Y := \bar V \subset \bP.$$
The variety $Y$ is called the {\bf Schubert variety} of $\cA$.
The translation action of $\Lambda$ on itself extends to an action on $\bP$,
and the subgroup $V\subset \Lambda$ acts on the subvariety $Y \subset\bP$.

For any subset $F\subset \cI$, let $e_F \in \bP$ be the point with coordinates
$$(e_F)_i = \begin{cases}
0 \;\;\;\;\,\text{if $i\in F$}\\
\infty\;\;\;\text{if $i\in F^c$},
\end{cases}$$
and let $$V_F := \{p\in Y\mid p_i =\infty\iff i\in F^c\}.$$
A subset $F\subset \cI$ is called a {\bf flat} if there exists a point
$v\in V$ such that $F = \{i\mid v\in H_i\}$.  Given a flat $F$, we define 
$$V^F := \bigcap_{i\in F} H_i.$$

\begin{proposition}\label{paving}
The variety $Y$ is stratified by affine spaces indexed by the flats of $\cA$.  More precisely:
\begin{enumerate}
\item For any subset $F\subset \cI$, $V_F\neq\emptyset\iff e_F \in Y \iff \text{$F$ is a flat.}$
\item For every flat $F$, $\operatorname{Stab}_V(e_F) = V^F$ and $V_F = V\cdot e_F\; \cong \; V/V^F$.
\item For every flat $G$, $\bar V_G = \bigcup_{F\subset G} V_F$.
\end{enumerate}
\end{proposition}

\begin{proof}
Item 1 is proved in \cite[Lemmas 7.5 and 7.6]{PXY}.
For the first part of item 2, we observe that $\operatorname{Stab}_V(e_F)$ is equal to the subgroup of $V\subset \Lambda$
consisting of elements $v$ that are supported on the set $\{i\mid (e_F)_i=\infty\} = F^c$.  This is equivalent to the condition
that $v\in H_i$ for all $i\in F$, in other words $v\in V^F$.  Thus the action of $V$ on $e_F$ defines an inclusion of $V/V^F$ into $V_F$.
The fact that this is an isomorphism follows from \cite[Lemma 7.6]{PXY}.  Item 3 is clear from the definition of $V_F$.
\end{proof}

We have a canonical action of $\gm$
on $\Lambda$ by scalar multiplication, which extends to an action on $\bP$ and restricts to a stratification-preserving action on $Y$.
For any flat $F\subset\cI$, let $$\bA^F:= \{p\in\bP\mid p_i = 0\iff i\in F\}.$$
This is isomorphic to a vector space of dimension $|F|$,
and the action of $\gm$ on $\bP$ restricts to the action of $\gm$ on $\bA^F$ by inverse scalar multiplication.
In particular, the coordinate ring of $\bA^F$ is non-negatively graded by the action of $\gm$, and the vanishing locus
of the ideal of positively graded elements is equal to $\{e_F\}$.
Let $$C_F := \bA^F \cap Y.$$
This is a closed $\gm$-equivariant subvariety of $\bA^F$ containing $e_F$, which implies that it is an affine cone with cone point $e_F$.
Let $$U_{FG} := C_F \cap V_G\and X_{FG} := C_F\cap \bar V_G.$$

\begin{proposition}\label{normal arrangement}
For all $F\subset G$, the restriction of $\IC_{\bar V_G}$ to $C_F(\Fqb)$ is isomorphic to
$\IC_{X_{FG}}$.
\end{proposition}

\begin{proof}
Fix the flat $F$, and choose a section $s:V_F\to V$ of the projection from $V$ to $V_F$.
The action map $\varphi_F:s(V_F)\times C_F \hookrightarrow V\times Y \to Y$ 
is an open immersion \cite[Section 3]{fs-braid},
thus we can apply Lemma \ref{orbit}.
\end{proof}



Let $L$ be the lattice of flats of $\cA$, ordered by inclusion.
If $F$ is a flat, the {\bf rank} of $F$ is defined to be the dimension of $V_F$, and we define a weak
rank function $r$ by putting $r_{FG} := \rk G - \rk F$ for all $F\leq G$.
Let $\chi\in\scrI(L)$ be the characteristic function (Example \ref{ex:char}).

\begin{proposition}\label{crapo}
For any pair of flats $F\leq G$ and any positive integer $s$, $\chi_{FG}(q^s) = |U_{FG}(\mathbb{F}_{q^s})|$.
\end{proposition}

\begin{proof}
When $F=\emptyset$ and $G=\cI$, $U_{\emptyset\cI} = V\smallsetminus \bigcup_{i\in\cI}H_i$ is equal to the complement of
the arrangement $\cA$ in $V$.  In this case, Crapo and Rota \cite[Section 16]{CrapoRota} prove that
$\chi_{\emptyset\cI}(q^s) = |U_{\emptyset\cI}(\mathbb{F}_{q^s})|$.

More generally, for any pair of flats $F\leq G$, consider the hyperplane arrangement
$$\cA^F_G := \{(H_i \cap V^F)/V^G\mid i\in G\smallsetminus F\}$$
in the vector space $V^F/V^G$.  The interval $[F,G]\subset L$ is isomorphic as a weakly ranked poset 
to the lattice of flats of $\cA^F_G$, and $U_{FG}$ is isomorphic to the complement of $\cA^F_G$ in $V^F/V^G$.
Thus Crapo and Rota's result, applied to the arrangement $\cA^F_G$, tells us that
$\chi_{FG}(q^s) = |U_{FG}(\mathbb{F}_{q^s})|$.
\end{proof}

The following result originally appeared in \cite[Theorem 3.10]{EPW}.

\begin{corollary}\label{arrangement-corollary}
Let $L$ be the weakly ranked poset of flats of the hyperplane arrangement $\cA$,
and let $f\in\scrI(L)$ be the right KLS-function associated with the $L$-kernel $\chi$.
For all $F\leq G\in L$, $\IH^{*}_{e_F}(\bar V_G)$ is chaste, and
$$f_{FG}(t) = \sum_{i\geq 0} t^i \dim \IH^{2i}_{e_F}(\bar V_G) = \sum_{i\geq 0} t^i \dim \IH^{2i}(X_{FG}).$$
\end{corollary}

\begin{proof}
This follows from Lemma \ref{cone point} and Theorem \ref{main} via Propositions \ref{paving}-\ref{crapo}.
\end{proof}

\begin{remark}
The variety $X_{FG}$ is called the {\bf reciprocal plane} of the arrangement $\cA^F_G$.  Its coordinate ring
is isomorphic to the {\bf Orlik-Terao algebra} of $\cA^F_G$, which is by definition the subalgebra of rational functions on $V^F/V^G$ generated
by the reciprocals of the linear forms that define the hyperplanes.
\end{remark}

\begin{remark}
By Proposition \ref{paving}(2), the strata of $Y$ are isomorphic to affine spaces, 
so Equations \eqref{catO1} and \eqref{catO2} tell us that $f_{xy}(t)$ may also be interpreted as the graded dimension
of an Ext group in category $\cO$, or as the graded multiplicity of a costandard in a simple in the Grothendieck group of the graded lift.
\end{remark}

Turning now to the $Z$-polynomial $Z\in\Ih(L)$ associated with $\chi$, we have the following corollary
of Theorem \ref{coh-Z0}.  A version of this result, along with the more general Theorem \ref{Z-arr}, 
originally appeared in \cite[Theorem 7.2]{PXY}.

\begin{corollary}\label{arr-Z0}
For all $F\in L$, $\IH^{*}(\bar V_F)$ is chaste, and $$Z_{\emptyset F}(t) = \sum_{i\geq 0}t^i \dim \IH^{2i}(\bar V_F).$$
\end{corollary}

\begin{proof}
As we noted in Example \ref{ex:char}, the $L$-kernel $\chi$ has left KLS-polynomial $\eta$,
and for all $F\in L$ and $s>0$, $|V_F(\Fqs)| = q^{sr_{\emptyset F}} = \bar \eta_{\emptyset F}(q^s)$.
Then Theorem \ref{coh-Z0} gives us our result.
\end{proof}

As in Section \ref{sec:flag}, we can give a cohomological interpretation of $Z_{FG}(t)$ for any $F\leq G\in L$.

\begin{theorem}\label{Z-arr}
For all $F\leq G\in L$, $\IH^{*}(\bar C_F\cap \bar V_G)$ is chaste, and $$Z_{FG}(t) = \sum_{i\geq 0}t^i \dim \IH^{2i}(\bar C_F\cap \bar V_G).$$
\end{theorem}

\begin{proof}
The variety $\bar C_F\cap \bar V_G$ is isomorphic to the variety $Y$ associated with the arrangement $\cA^F_G$.
Similarly, the interval $[F,G]\subset L$ is isomorphic as a weakly ranked poset to the lattice of flats of $\cA^F_G$.
Thus the theorem follows from Corollary \ref{arr-Z0} applied to the arrangement $\cA^F_G$
and the pair of flats $\emptyset \leq G\smallsetminus F$.
\end{proof}

\begin{remark}
We have chosen to work with arrangements over a finite field in order to apply the techniques
of Section \ref{sec:geometry}, but this restriction is not important.  First, given a hyperplane arrangement
over any field, it is possible to choose a combinatorially equivalent arrangement (one with the same
matroid) over a finite field \cite[Theorems 4 \& 6]{Rado}.  Second, if we are given an arrangement over
the complex numbers and we prefer to work with the topological intersection cohomology
of the analogous complex varieties, the formulas in the statements of
Corollary \ref{arrangement-corollary} and Theorem \ref{Z-arr} still
hold (see \cite[Proposition 3.12]{EPW} and \cite[Theorem 7.2]{PXY}).
\end{remark}

\begin{remark}\label{different}
The proof of Theorem \ref{richardson} (the analogue of Theorem \ref{Z-arr} for Richardson varieties)
relied on two special facts, namely Equations \eqref{opp} and \eqref{KWY}.
In the context of hyperplane arrangements, the analogues of these two equations hold {\em a posteriori},
but it is not clear how one would prove them directly.  In particular, the variety $C_F$ is not smooth, so 
the decomposition $Y = \bigsqcup_{F\in L} C_F$ is not a stratification, and it is not possible
to apply Theorem \ref{main} to obtain the analogue of Equation \eqref{opp}.
On the other hand, the proof of Theorem \ref{Z-arr} relies on the fact that any interval in the lattice
of flats of an arrangement is isomorphic to the lattice of flats of another arrangement; the analogous
statement for the Bruhat order on a Coxeter group is false.  Thus the proofs of Theorems \ref{richardson}
and \ref{Z-arr} are truly distinct.
\end{remark}

\subsection{Toric varieties}\label{sec:toric}
Let $T$ be a split algebraic torus over $\Fq$ with cocharacter lattice $N$
and let $\Sigma$ be a rational fan in $N_\R$.  We consider $\Sigma$ to be a weakly ranked poset ordered by {\em reverse} inclusion, with weak rank function given by relative dimension.  We will assume that $\{0\}\in \Sigma$;
this is the {\em maximal} element of $\Sigma$, and we will denote it simply by 0.

Let $Y$ be the $T$-toric variety associated with $\Sigma$.
The cones of $\Sigma$ are in bijection with $T$-orbits in $Y$ and with $T$-invariant affine open subsets of $Y$.
Given $\sigma\in\Sigma$, let $V_\sigma$ denote the corresponding orbit, let $W_\sigma$ denote the corresponding
affine open subset, and let $T_\sigma\subset T$ be the stabilizer of 
any point in $V_\sigma$.  We then have $\dim V_\sigma = \codim \sigma$, and \cite[Theorem 3.2.6]{CLS}
$$\sigma\leq\tau\iff V_\sigma \subset \bar V_\tau \iff W_\sigma\supset W_\tau\iff W_\sigma\supset V_\tau.$$
For each $\sigma\in\Sigma$, we have a canonical identification $V_\sigma \cong T/T_\sigma$, and we define $e_\sigma\in V_\sigma$
to be the identity element of $T/T_\sigma$.  In particular, we have $T_\sigma\subset T \cong V_0 \subset Y$ for all $\sigma$, and we define
$$C_\sigma := W_\sigma\cap\bar T_\sigma.$$
The cocharacter lattice of $T_\sigma$ is equal to $N_\sigma := N\cap \R\sigma$, $C_\sigma$
is isomorphic to the $T_\sigma$-toric variety associated with the cone $\sigma\subset N_{\sigma,\R}$,
and $e_\sigma\in C_\sigma$ is the unique fixed point.  If $\sigma\leq\tau$, then 
$U_{\sigma\tau} := C_\sigma\cap V_\tau$ is equal to the $T_\sigma$-orbit in $C_\sigma$ corresponding
to the face $\tau$ of $\sigma$.  In particular, this means that 
$$|U_{\sigma\tau}(\Fqs)| = (q^s-1)^{r_{\sigma\tau}} = \lambda_{\sigma\tau}(q^s),$$
where $\la\in\scrI(\Sigma)$ is the $\Sigma$-kernel of Example \ref{ex:Eulerian}.

For each $\sigma\in\Sigma$,
choose a lattice point $n_\sigma\in N$ lying in the relative interior of $\sigma$.  Then $n_\sigma$ is a cocharacter
of $T$, and thus defines a homomorphism $\rho_\sigma:\gm\to T\subset\Aut(Y)$.  The fact that $\sigma$ lies
in the relative interior of $\sigma$ implies that $C_\sigma$ is a weighted affine cone with respect to $\rho_\sigma$
with cone point $e_\sigma$.  Choose in addition a section $s_\sigma:T/T_\sigma\to T$ of the projection.
Then the action map $s_\sigma(T/T_\sigma)\times C_\sigma\to Y$ is an open immersion, thus 
Lemma \ref{orbit} tells us that the hypotheses of Section \ref{sec:setup} are satisfied.  We therefore obtain
the following corollary to Theorem \ref{main}, which originally appeared in \cite[Theorem 6.2]{DL}
(see also \cite[Theorem 1.2]{Fieseler}).

\begin{corollary}\label{f-sigma}
Let $f\in\Ih(\Sigma)$ be the right KLS-function associated with $\la$.  For all $\sigma\leq\tau$,
$\IH^*_{e_\sigma}(\bar V_\tau)$ is chaste and
$$\sum_{i\geq 0}t^i \dim \IH^{2i}_{e_\sigma}(\bar V_\tau) = f_{\sigma\tau}(t).$$
\end{corollary}

\begin{remark}
Let $\Delta$ be a lattice polytope, and let $\Sigma$ be the fan consisting of the cone
over $\Delta$ along with all of its faces.  Then $\Sigma$, ordered by reverse inclusion,
is isomorphic to the opposite of the face poset of $\Delta$, ordered by inclusion.
It follows from Remark \ref{opposite} that, if $g\in\Ih(\Delta) \cong \Ih(\Sigma^*)$ is the left KLS-function
associated with the Eulerian poset of faces of $\Delta$, then $g^* = f\in \Ih(\Sigma)$.
In particular, the $g$-polynomial $g_{\emptyset\Delta}(t)$ is equal to $f_{c\Delta 0}(t)$.
\end{remark}

\subsection{Hypertoric varieties}\label{sec:hypertoric}
Let $N$ be a finite dimensional lattice and let $\gamma := (\gamma_i)_{i\in \cI}$ 
be an $\cI$-tuple of nonzero elements of $N$ that together span a cofinite sublattice of $N$.
Then $\gamma$ defines a homomorphism from $\Z^\cI$ to $N$, along with
a dual inclusion from $N^*$ to $\Z^\cI$.  As in Section \ref{sec:arrangement}, we define a subset
$F\subset\cI$ to be a {\bf flat} if there exists an element $m\in N^*\subset\Z^\cI$ such that 
$m_i = 0\iff i\in F$.  Given a flat $F$, we let $\gamma_F := (\gamma_i)_{i\in F}$ and 
we define $N_F \subset N$ to be the saturation of the span of $\gamma_F$.  
We also define $N^F:=N/N_F$, and we define $\gamma^F$ to be the image of $(\gamma_i)_{i\notin F}$
in $N^F$.

Choose a prime power $q$ with the property that, for any subset $\cJ\subset\cI$, the multiset
$\{\gamma_i\mid i\in\cJ\}$ is linearly independent only if its image in $N_{\Fq}$ is linearly independent.
Let $Q := \Fq[z_i,w_i]_{i\in\cI}$.  This ring admits a grading by the group $\Z^\cI = \Z\{x_i\mid i\in\cI\}$
in which $\deg z_i = -\deg w_i = x_i$.  The degree zero part $Q_0 = \Fq[z_iw_i]_{i\in\cI}$
maps to $\Sym N_{\Fq}$ by sending $z_iw_i$ to the reduction modulo $q$ of $\gamma_i$.  Let $Q_{N^*}$
be the subring of $Q$ with basis consisting of $\Z^\cI$-homogeneous elements whose degrees\
lie in $N^*\subset\Z^\cI$, and let $R:=Q_{N^*}\otimes_{Q_0}\Sym N_{\Fq}$.
The variety $Y = Y(\gamma) := \Spec R$ is called a {\bf hypertoric variety}.

Let $\becircled Y\subset Y$ be the open subvariety defined by the nonvanishing of all elements of $R$ that lift
to monomials in $Q$.   Let $L$ be the lattice of flats of $\gamma$.  We have a stratification 
$$Y = \bigsqcup_{F\in L} V_F,$$
with the property that $V_F \cong \becircled Y(\gamma^F)$ \cite[Equation 5]{PW07}.
In particular, the largest stratum is $V_\emptyset$ and the smallest stratum is $V_{\cI}$.
More generally, the partial order induced by the stratification is the opposite of the inclusion order.
For any $F\subset G$, the dimension of $V_F$ minus the dimension of $V_G$ is equal to $2r_{FG}$,
where $r$ is the usual weak rank function (as in Example \ref{ex:hypertoric}).

At this point, we are forced to depart from the setup of Section \ref{sec:setup}.
We are supposed to define a subvariety $C_F\subset Y$ for each flat $F$, satisfying certain properties;
then for every $F\subset G$, we would consider the varieties $U_{GF} = C_G\cap V_F$ and 
$X_{GF} = C_G\cap \bar V_F$.  Morally, we should have $C_F \cong Y(\gamma_F)$, 
$X_{GF}\cong Y(\gamma^F_G)$, and $U_{GF}\cong \becircled Y(\gamma^F_G)$.
Unfortunately, we do not know of any natural way to embed $Y(\gamma_F)$ into $Y$ to achieve these isomorphisms.
Instead, we will simply define $X_{GF}$ and $U_{GF}$ as above.  The conclusion of Lemma \ref{X-strat}
clearly holds for this definition, while the conclusion of Lemma \ref{slice} follows from \cite[Lemma 2.4]{PW07}.
Thus Theorem \ref{main} still holds as stated.  By \cite[Proposition 4.2]{PW07},
for all $s>0$ and all flats $F\subset G$, we have $|U_{GF}(\Fqs)| = \kappa_{FG}(q^s)$,
where $\kappa\in\scrI(L,2r)$ is the $(L,2r)$-kernel of Example \ref{ex:hypertoric}.

\begin{corollary}
Let $\hbc\in\scrI(L,2r)$ be the left KLS-function associated with the $(L,2r)$-kernel 
$\kappa$ of Example \ref{ex:hypertoric}.
For all flats $F\subset G\in L$, $\IH^{*}(\bar X_{GF})$ is chaste, and
$$\hbc_{FG}(t) = \sum_{i\geq 0} t^i \dim \IH^{2i}(\bar X_{GF}) .$$
\end{corollary}

\begin{proof}
As noted above, our stratification of $Y$ induces the weakly ranked poset $(L^*,2r^*)$.
Let $f$ be the right KLS-function associated with the $(L^*,2r^*)$-kernel $\kappa^*$.
For all $s>0$ and all flats $F\subset G$, we have 
$\kappa^*_{GF}(q^s) = \kappa_{FG}(q^s) = |U_{GF}(\Fqs)|$, thus Theorem \ref{main}
tells us that $\IH^{*}(\bar X_{GF})$ is chaste, and
$$f_{GF}(t) = \sum_{i\geq 0} t^i \dim \IH^{2i}(\bar X_{GF}) .$$
By Remark \ref{opposite}, we have $\hbc = f^*$, which proves the corollary.
\end{proof}

\bibliography{./symplectic}

\def\cprime{$'$}
\providecommand{\bysame}{\leavevmode\hbox to3em{\hrulefill}\thinspace}
\providecommand{\MR}{\relax\ifhmode\unskip\space\fi MR }
\providecommand{\MRhref}[2]{%
  \href{http://www.ams.org/mathscinet-getitem?mr=#1}{#2}
}
\providecommand{\href}[2]{#2}
\begin{thebibliography}{KWY13}

\bibitem[BB96]{Bat-Bor}
Victor~V. Batyrev and Lev~A. Borisov, \emph{Mirror duality and string-theoretic
  {H}odge numbers}, Invent. Math. \textbf{126} (1996), no.~1, 183--203.

\bibitem[BF10]{BF}
Alexander Braverman and Michael Finkelberg, \emph{Pursuing the double affine
  {G}rassmannian. {I}. {T}ransversal slices via instantons on
  {$A_k$}-singularities}, Duke Math. J. \textbf{152} (2010), no.~2, 175--206.

\bibitem[BGS96]{BGS96}
Alexander Beilinson, Victor Ginzburg, and Wolfgang Soergel, \emph{Koszul
  duality patterns in representation theory}, J. Amer. Math. Soc. \textbf{9}
  (1996), no.~2, 473--527.

\bibitem[Bra06]{Braden-CICF}
Tom Braden, \emph{Remarks on the combinatorial intersection cohomology of
  fans}, Pure Appl. Math. Q. \textbf{2} (2006), no.~4, Special Issue: In honor
  of Robert D. MacPherson. Part 2, 1149--1186.

\bibitem[Bre99]{Brenti-twisted}
Francesco Brenti, \emph{Twisted incidence algebras and
  {K}azhdan-{L}usztig-{S}tanley functions}, Adv. Math. \textbf{148} (1999),
  no.~1, 44--74.

\bibitem[Bre03]{Brenti-IC}
\bysame, \emph{{P}-kernels, {IC} bases and {K}azhdan--{L}usztig polynomials},
  Journal of Algebra \textbf{259} (2003), no.~2, 613--627.

\bibitem[CLS11]{CLS}
David~A. Cox, John~B. Little, and Henry~K. Schenck, \emph{Toric varieties},
  Graduate Studies in Mathematics, vol. 124, American Mathematical Society,
  Providence, RI, 2011.

\bibitem[CPS93]{CPS-abstract}
Edward Cline, Brian Parshall, and Leonard Scott, \emph{Abstract
  {K}azhdan-{L}usztig theories}, Tohoku Math. J. (2) \textbf{45} (1993), no.~4,
  511--534.

\bibitem[CPS97]{CPS-graded}
\bysame, \emph{Graded and non-graded {K}azhdan-{L}usztig theories}, Algebraic
  groups and {L}ie groups, Austral. Math. Soc. Lect. Ser., vol.~9, Cambridge
  Univ. Press, Cambridge, 1997, pp.~105--125.

\bibitem[CR70]{CrapoRota}
Henry~H. Crapo and Gian-Carlo Rota, \emph{On the foundations of combinatorial
  theory: {C}ombinatorial geometries}, preliminary ed., The M.I.T. Press,
  Cambridge, Mass.-London, 1970.

\bibitem[dCM09]{dCM}
Mark Andrea~A. de~Cataldo and Luca Migliorini, \emph{The decomposition theorem,
  perverse sheaves and the topology of algebraic maps}, Bull. Amer. Math. Soc.
  (N.S.) \textbf{46} (2009), no.~4, 535--633.

\bibitem[DL91]{DL}
J.~Denef and F.~Loeser, \emph{Weights of exponential sums, intersection
  cohomology, and {N}ewton polyhedra}, Invent. Math. \textbf{106} (1991),
  no.~2, 275--294.

\bibitem[Du94]{Du}
Jie Du, \emph{{${\rm IC}$} bases and quantum linear groups}, Algebraic groups
  and their generalizations: quantum and infinite-dimensional methods
  ({U}niversity {P}ark, {PA}, 1991), Proc. Sympos. Pure Math., vol.~56, Amer.
  Math. Soc., Providence, RI, 1994, pp.~135--148.

\bibitem[Dye93]{Dyer}
M.~J. Dyer, \emph{Hecke algebras and shellings of {B}ruhat intervals},
  Compositio Math. \textbf{89} (1993), no.~1, 91--115.

\bibitem[EPW16]{EPW}
Ben Elias, Nicholas Proudfoot, and Max Wakefield, \emph{The {K}azhdan-{L}usztig
  polynomial of a matroid}, Adv. Math. \textbf{299} (2016), 36--70.

\bibitem[EW14]{EW14}
Ben Elias and Geordie Williamson, \emph{The {H}odge theory of {S}oergel
  bimodules}, Ann. of Math. (2) \textbf{180} (2014), no.~3, 1089--1136.

\bibitem[Fie91]{Fieseler}
Karl-Heinz Fieseler, \emph{Rational intersection cohomology of projective toric
  varieties}, J. Reine Angew. Math. \textbf{413} (1991), 88--98.

\bibitem[GPY17]{kl-survey}
Katie Gedeon, Nicholas Proudfoot, and Benjamin Young, \emph{Kazhdan-{L}usztig
  polynomials of matroids: a survey of results and conjectures}, S\'em. Lothar.
  Combin. \textbf{78B} (2017), Art. 80, 12.

\bibitem[Hai]{Haines}
Thomas Haines, \emph{A proof of the {K}azhdan-{L}usztig purity theorem via the
  decomposition theorem of {BBD}},
  \textsf{http://www.math.umd.edu/$\sim$tjh/KL\_purity1.pdf}.

\bibitem[Kar04]{Karu}
Kalle Karu, \emph{Hard {L}efschetz theorem for nonrational polytopes}, Invent.
  Math. \textbf{157} (2004), no.~2, 419--447.

\bibitem[KL79]{KL79}
David Kazhdan and George Lusztig, \emph{Representations of {C}oxeter groups and
  {H}ecke algebras}, Invent. Math. \textbf{53} (1979), no.~2, 165--184.

\bibitem[KL80]{KL80}
David Kazhdan and George Lusztig, \emph{Schubert varieties and {P}oincar{\'e}
  duality}, Geometry of the Laplace operator (Proc. Sympos. Pure Math., Univ.
  Hawaii, Honolulu, Hawaii, 1979), Proc. Sympos. Pure Math., XXXVI, Amer. Math.
  Soc., Providence, R.I., 1980, pp.~185--203.

\bibitem[KW01]{KW}
Reinhardt Kiehl and Rainer Weissauer, \emph{Weil conjectures, perverse sheaves
  and {$l$}'adic {F}ourier transform}, Ergebnisse der Mathematik und ihrer
  Grenzgebiete. 3. Folge. A Series of Modern Surveys in Mathematics [Results in
  Mathematics and Related Areas. 3rd Series. A Series of Modern Surveys in
  Mathematics], vol.~42, Springer-Verlag, Berlin, 2001.

\bibitem[KW06]{KirwanWoolf}
Frances Kirwan and Jonathan Woolf, \emph{An introduction to intersection
  homology theory}, second ed., Chapman \& Hall/CRC, Boca Raton, FL, 2006.

\bibitem[KWY13]{KWY}
Allen Knutson, Alexander Woo, and Alexander Yong, \emph{Singularities of
  {R}ichardson varieties}, Math. Res. Lett. \textbf{20} (2013), no.~2,
  391--400.

\bibitem[Lus83]{Lusztig-singularities}
George Lusztig, \emph{Singularities, character formulas, and a {$q$}-analog of
  weight multiplicities}, Analysis and topology on singular spaces, {II}, {III}
  ({L}uminy, 1981), Ast\'erisque, vol. 101, Soc. Math. France, Paris, 1983,
  pp.~208--229.

\bibitem[Lus03]{Lusztig-unequal}
G.~Lusztig, \emph{Hecke algebras with unequal parameters}, CRM Monograph
  Series, vol.~18, American Mathematical Society, Providence, RI, 2003.

\bibitem[NR03]{NelsenRam}
Kendra Nelsen and Arun Ram, \emph{Kostka-{F}oulkes polynomials and {M}acdonald
  spherical functions}, Surveys in combinatorics, 2003 ({B}angor), London Math.
  Soc. Lecture Note Ser., vol. 307, Cambridge Univ. Press, Cambridge, 2003,
  pp.~325--370.

\bibitem[Pol99]{Polo}
Patrick Polo, \emph{Construction of arbitrary {K}azhdan-{L}usztig polynomials
  in symmetric groups}, Represent. Theory \textbf{3} (1999), 90--104.

\bibitem[PW07]{PW07}
Nicholas Proudfoot and Ben Webster, \emph{Intersection cohomology of hypertoric
  varieties}, J. Algebraic Geom. \textbf{16} (2007), no.~1, 39--63.

\bibitem[PXY18]{PXY}
Nicholas Proudfoot, Yuan Xu, and Ben Young, \emph{The {$Z$}-polynomial of a
  matroid}, Electron. J. Combin. \textbf{25} (2018), no.~1, Paper 1.26, 21.

\bibitem[PY17]{fs-braid}
Nicholas Proudfoot and Ben Young, \emph{Configuration spaces, {$\rm
  FS^{op}$}-modules, and {K}azhdan-{L}usztig polynomials of braid matroids},
  New York J. Math. \textbf{23} (2017), 813--832.

\bibitem[Rad57]{Rado}
R.~Rado, \emph{Note on independence functions}, Proc. London Math. Soc. (3)
  \textbf{7} (1957), 300--320.

\bibitem[Sta80]{S1}
Richard~P. Stanley, \emph{The number of faces of a simplicial convex polytope},
  Adv. in Math. \textbf{35} (1980), no.~3, 236--238.

\bibitem[Sta87]{Stanley-IC}
Richard Stanley, \emph{Generalized {$H$}-vectors, intersection cohomology of
  toric varieties, and related results}, Commutative algebra and combinatorics
  ({K}yoto, 1985), Adv. Stud. Pure Math., vol.~11, North-Holland, Amsterdam,
  1987, pp.~187--213.

\bibitem[Sta92]{Stanley-h}
Richard~P. Stanley, \emph{Subdivisions and local {$h$}-vectors}, J. Amer. Math.
  Soc. \textbf{5} (1992), no.~4, 805--851.

\bibitem[Wak18]{Wak}
Max Wakefield, \emph{A flag {W}hitney number formula for matroid
  {K}azhdan-{L}usztig polynomials}, Electron. J. Combin. \textbf{25} (2018),
  no.~1, Paper 1.22, 14.

\bibitem[Zhu]{Zhu}
Xinwen Zhu, \emph{An introduction to affine {G}rassmannians and the geometric
  {S}atake equivalence}, \textsf{ arXiv:1603.05593}.

\end{thebibliography}
\bibliographystyle{amsalpha}

\end{document}